\documentclass[11pt,a4paper]{amsart}
\usepackage[T1]{fontenc}
\usepackage[utf8]{inputenc}
\usepackage{amssymb,url,upref,verbatim,xspace,mathrsfs}
\usepackage{mathtools}
\mathtoolsset{showonlyrefs}
\RequirePackage[%
pdfview=XYZ,
pdfstartview=FitV,
pdffitwindow,
hyperfootnotes=false,
raiselinks=true,
hypertexnames=false,
colorlinks=true,
bookmarks=true,
bookmarksopen=false,
anchorcolor=black,
citecolor=black,
linkcolor=black,
filecolor=black,
urlcolor=black]{hyperref}

\RequirePackage[all,ps,cmtip]{xy}
\newdir^{ (}{{}*!/-5pt/\dir^{(}}
\newdir_{ (}{{}*!/-5pt/\dir_{(}}
\RequirePackage[mathscr]{eucal}
\usepackage{mathrsfs}\let\mathcal\mathscr
\newcounter{toto}
\def\thetoto{\arabic{toto}}
\let\oldmarginpar\marginpar
\def\marginpar#1{\refstepcounter{toto}\textsuperscript{\textup{[\thetoto]}}\oldmarginpar{\footnotesize\textsuperscript{[\thetoto]}\,#1}}

\theoremstyle{plain}
\newtheorem{theo}{Theorem}
\newtheorem{coro}[theo]{Corollary}

\newtheorem{theorem}{Theorem}[section]
\newtheorem{proposition}[theorem]{Proposition}
\newtheorem{lemma}[theorem]{Lemma}
\newtheorem{corollary}[theorem]{Corollary}

\theoremstyle{definition}
\newtheorem{definition}[theorem]{Definition}
\newtheorem{remark}[theorem]{Remark}
\newtheorem{example}[theorem]{Example}
\newtheorem*{convnot}{Convention and notation}

\def\mainmatter{\renewcommand{\baselinestretch}{1.1}\normalfont}
\def\backmatter{\renewcommand{\baselinestretch}{1}\normalfont}

\makeatletter
\@namedef{subjclassname@2010}{\textup{2010} Mathematics Subject Classification}
\def\l@section{\@tocline{1}{0pt}{0pc}{}{}}
\def\l@subsection{\@tocline{2}{0pt}{1.5pc}{}{}}
\def\l@subsubsection{\@tocline{3}{0pt}{2pc}{}{}}
\makeatother

\newcommand{\C}{\mathbb{C}}\let\CC\C
\newcommand{\N}{\mathbb{N}}\let\NN\N
\newcommand{\R}{\mathbb{R}}

\let\ZZ\Z
\newcommand{\perv}{\mathrm{perv}}
\newcommand{\bD}{\boldsymbol{D}}
\newcommand{\shhom}{\mathcal{H}\!\mathit{om}}\let\ho\shhom
\DeclareMathOperator{\rh}{\mathit{R}\shhom}\let\Rhom\rh
\DeclareMathOperator{\tho}{\mathit{T}\shhom}
\DeclareMathOperator{\RH}{RH}
\let\TH\THH
\newcommand{\rb}{\mathrm{b}}
\newcommand{\rd}{\mathrm{d}}
\newcommand{\lc}{\mathrm{lc}}

\newcommand{\coh}{\mathrm{coh}}
\newcommand{\hol}{\mathrm{hol}}
\newcommand{\rhol}{\mathrm{rhol}}
\newcommand{\Mod}{\mathsf{Mod}}
\newcommand{\sa}{\mathrm{sa}}
\newcommand{\tf}{\mathrm{tf}}
\newcommand{\rt}{\mathrm{t}}
\newcommand{\rhog}{\boldsymbol{\rho}}
\renewcommand{\thetag}{\boldsymbol{\theta}}
\newcommand{\cc}{{\C\textup{-c}}}
\newcommand{\rc}{{\R\textup{-c}}}
\newcommand{\XS}{X\times S}
\newcommand{\XT}{X\times T}
\newcommand{\XpS}{X'\times S}
\newcommand{\YpS}{Y'\times\nobreak S}
\newcommand{\XsS}{X^*\times S}
\newcommand{\YS}{Y\times\nobreak S}

\newcommand{\US}{U\times S}
\newcommand{\DX}{\shd_X}
\newcommand{\DXS}{\shd_{\XS/S}}

\newcommand{\DXpS}{\shd_{\XpS/S}}

\newcommand{\DYS}{\shd_{Y\times S/S}}
\newcommand{\DUS}{\shd_{U\times S/S}}
\newcommand{\DXT}{\shd_{X\times T/T}}

\DeclareMathOperator{\Char}{Char}

\DeclareMathOperator{\codim}{codim}

\DeclareMathOperator{\pD}{{}^\mathrm{p}\mathsf{D}}
\DeclareMathOperator{\rD}{\mathsf{D}}

\DeclareMathOperator{\DR}{DR}
\DeclareMathOperator{\Db}{\mathfrak{Db}}

\DeclareMathOperator{\Hom}{Hom}
\DeclareMathOperator{\id}{Id}\let\Id\id
\let\im\imm
\DeclareMathOperator{\Op}{Op}

\DeclareMathOperator{\Sol}{Sol}
\DeclareMathOperator{\pSol}{{}^\mathrm{p}Sol}

\DeclareMathOperator{\supp}{Supp}
\let\tilde\widetilde
\let\wt\widetilde

\let\epsilon\varepsilon

\let\setminus\smallsetminus
\let\leq\leqslant
\let\geq\geqslant

\def\loccit{loc.\kern3pt cit.{}\xspace}
\def\cf{cf.\kern.3em}
\def\eg{e.g.\kern.3em}
\def\ie{i.e.,\ }
\def\resp{\text{resp.}\kern.3em}
\let\moins\smallsetminus
\let\oldvee\vee
\renewcommand{\vee}{{\scriptscriptstyle\oldvee}}

\newcommand{\wtj}{\wt\jmath}
\newcommand{\Df}{{}_{\scriptscriptstyle\mathrm{D}}f}

\newcommand{\Di}{{}_{\scriptscriptstyle\mathrm{D}}i}
\newcommand{\Ddelta}{{}_{\scriptscriptstyle\mathrm{D}}\delta}
\newcommand{\Dpi}{{}_{\scriptscriptstyle\mathrm{D}}\pi}
\newcommand{\cbbullet}{{\raisebox{1pt}{$\sbullet$}}}
\newcommand{\sbullet}{{\scriptscriptstyle\bullet}}
\newcommand{\pOS}{p^{-1}\sho_S}

\newcommand{\pOXS}{p^{-1}_X\sho_S}
\newcommand{\pOXT}{p^{-1}_X\sho_T}
\newcommand{\pOYS}{p^{-1}_Y\sho_S}

\def\sha{\mathcal{A}}

\def\shc{\mathcal{C}}
\def\shd{\mathcal{D}}

\def\shf{\mathcal{F}}\let\cF F
\def\shg{\mathcal{G}}\let\cG G
\def\shh{\mathcal{H}}

\def\shl{\mathcal{L}}
\def\shm{\mathcal{M}}
\def\shn{\mathcal{N}}
\def\sho{\mathcal{O}}

\def\shs{\mathcal{S}}
\def\sht{\mathcal{T}}

\newcommand{\RedefinitSymbole}[1]{%
\expandafter\let\csname old\string#1\endcsname=#1
\let#1=\relax
\newcommand{#1}{\csname old\string#1\endcsname\,}%
}
\RedefinitSymbole{\forall} \RedefinitSymbole{\exists}

\def\to{\mathchoice{\longrightarrow}{\rightarrow}{\rightarrow}{\rightarrow}}
\def\mto{\mathchoice{\longmapsto}{\mapsto}{\mapsto}{\mapsto}}
\def\hto{\mathrel{\lhook\joinrel\to}}

\def\To#1{\mathchoice{\xrightarrow{\textstyle\kern4pt#1\kern3pt}}{\stackrel{#1}{\longrightarrow}}{}{}}

\def\isom{\stackrel{\sim}{\longrightarrow}}

\let\oldbigoplus\bigoplus
\renewcommand{\bigoplus}{\mathop{\textstyle\oldbigoplus}\displaylimits}
\let\oldbigwedge\bigwedge
\renewcommand{\bigwedge}{\mathop{\textstyle\oldbigwedge}\displaylimits}
\let\oldbigcap\bigcap
\renewcommand{\bigcap}{\mathop{\textstyle\oldbigcap}\displaylimits}
\let\oldbigcup\bigcup
\renewcommand{\bigcup}{\mathop{\textstyle\oldbigcup}\displaylimits}

\begin{document}
\author{Luisa Fiorot}
\address[L. Fiorot]{Dipartimento di Matematica ``Tullio Levi-Civita'' Universit\`a degli Studi di Padova\\
Via Trieste, 63\\
35121 Padova Italy}
\email{luisa.fiorot@unipd.it}

\author{Teresa Monteiro Fernandes}
\address[T. Monteiro Fernandes]{Centro de Matem\'atica e Aplica\c{c}\~{o}es Funda\-men\-tais-CIO and Departamento de Matem\' atica da Faculdade de Ci\^en\-cias da Universidade de Lisboa, Bloco C6, Piso 2, Campo Grande, 1749-016, Lisboa
Portugal}
\email{mtfernandes@fc.ul.pt}

\author{Claude Sabbah}
\address[C.~Sabbah]{CMLS, CNRS, École polytechnique, Institut Polytechnique de Paris, 91128 Palaiseau cedex, France}
\email{Claude.Sabbah@polytechnique.edu}

\title[Relative regular Riemann-Hilbert correspondence]{Relative Regular Riemann-Hilbert correspondence}

\thanks{The research of L.\,Fiorot was supported by project PRIN 2017 Categories, Algebras: Ring-Theoretical and Homological Approaches (CARTHA). The research of T.\,Monteiro Fernandes was supported by Funda\c c\~ao para a Ci\^encia e a Tecnologia, UID/MAT/04561/2019.}

\keywords{relative $\mathcal D$-module, De Rham functor, regular holonomic $\mathcal D$-module}

\subjclass[2010]{14F10, 32C38, 35A27, 58J15}

\begin{abstract}
On the product of a complex manifold $X$ by a complex curve $S$ considered as a parameter space, we show a Riemann-Hilbert correspondence between regular holonomic relative $\mathcal D$-modules (resp.\ complexes) on the one hand and relative perverse complexes (resp.\ $S$\nobreakdash-$\mathbb{C}$-constructible complexes) on the other hand.
\end{abstract}

\maketitle
\tableofcontents
\mainmatter
\section*{Introduction}

Let $X$ and $T$ be complex manifolds. Relative $\shd$-modules are modules over the sheaf $\DXT$ of differential operators relative to the projection $p_X:\XT\to T$. \emph{Holonomic} $\DXT$-modules encode holomorphic families of holonomic $\shd_X$-modules parametrized by $T$ whose characteristic variety is contained in a fixed Lagrangian subset $\Lambda\subset T^*X$. This is a strong condition avoiding confluence phenomena. It describes nevertheless the generic behaviour of a deformation of a holonomic $\shd_X$-module. From this point of view, it is natural to emphasize those $\DXT$-modules which are $\pOXT$-flat, that we call \emph{strict}.

There exists a solution functor from the category of \emph{holonomic} $\DXT$-modules to that of perverse $\CC$-constructible complexes of $\pOXT$-modules, objects defined in \cite{MFCS1}. Strictness (\ie flatness) on the $\DXT$-module side corresponds to the property that the Verdier dual of a perverse complex (as~a complex of $\pOXT$-modules) is also perverse.

In their previous work \cite{MFCS2}, Monteiro Fernandes and Sabbah have introduced the notion of relative regular holonomic $\DXT$-modules. For example, if $T=\CC^*$, a $\DXT$-module underlying a regular mixed twistor $\shd$-module (\cf \cite{Mochizuki11}) is regular holonomic in this sense.

\begin{convnot}
In this article, we mainly consider the case where the parameter space $T$ has dimension one, and we emphasize this assumption by denoting it by $S$. Therefore, throughout this article, $S$ will denote a complex curve. On the other hand, we will set $d_X=\dim X$.
\end{convnot}

Our main result is a Riemann-Hilbert correspondence for regular holonomic $\DXS$-modules, in the following form. We let $\rD^{\rb}_{\rhol}(\DXS)$ denote the full subcategory of $\rD^{\rb}(\DXS)$ consisting of complexes having regular holonomic cohomology modules (\cf Section \ref{S1d} for details) and $\rD^\rb_{\cc}(\pOXS)$ the category of $\CC$-constructible complexes of $\pOXS$-modules (\cf Section \ref {subsec:prelimtop} and \cite{MFCS1}).

\begin{theo} \label{RHH}
The functors
\begin{align*}
\pSol_X: \rD^{\rb}_{\rhol}(\DXS)&\to \rD^\rb_{\cc}(\pOXS) \\
\RH^S_X:\rD^\rb_{\cc}(\pOXS)&\to \rD^{\rb}_{\rhol}(\DXS)
\end{align*}
are quasi-inverse equivalences of categories.
\end{theo}

The functor $\pSol_X$ is the solution functor shifted by the dimension of $X$ (\cf\cite[\S3.3]{MFCS1} and Section \ref{subsec:solfunctor}), and $\RH^S_X$ is the relative Riemann-Hilbert functor (\cf\cite[\S3.4]{MFCS2} and Section~\ref{S2}). One direction of the correspondence, namely $\id\isom\pSol_X\circ\RH^S_X$, was proved in \cite[Th.\,3]{MFCS2}, and a particular case of this correspondence was obtained as Theorem~5 in \loccit, namely, if $\shm$ underlies a regular mixed twistor $\shd$-module, then $\shm$ can be recovered from $\pSol_X(\shm)$ up to isomorphism by the formula $\shm\simeq\RH^S_X(\pSol_X(\shm))$.

The methods used in the present paper rely on the previous works \cite{MFCS1}, \cite{MFCS2} as well as \cite{SS1}, \cite{TL}, \cite{FMF1}. The main tools in the proof given in \cite{MFCS2} are the good functorial properties satisfied by holonomic $\shd$-modules underlying mixed twistor $\shd$-modules, which include stability under pullback, localization along an hypersurface and direct image by projective morphisms.

The main problem to extend these results to more general situations is the bad behaviour of $\DXS$-holonomicity by pullback in general (\cite[Ex.\,2.4]{MFCS3}). We realized however that we can avoid these general arguments for \emph{regular} holonomic $\DXS$-modules.
We replace them by proving that regular holonomiciy behaves well with respect to pullback:

\begin{theo}\label{inverseimage}
Let $\shm\in\rD^\rb_{\rhol}(\DXS)$ and let $f:Y\to X$ be a morphism of complex manifolds, then $\Df^*\shm\in\rD^\rb_{\rhol}(\DYS)$.
\end{theo}

Let us indicate the main points in the proof of Theorem~\ref{RHH} in Section~\ref{S3}. The first tool is \cite[Th.\,3]{MFCS2} which asserts that there exists a natural transformation
\[
\Id_{\rD^\rb_{\cc}(\pOXS)} \To{\alpha}\pSol_X \circ \RH^S_X
\]
providing a functorial isomorphism
\[
F\xrightarrow[\sim]{~\textstyle\alpha_F~} \pSol_X(\RH^S_X(F))
\]
for any $F\in \rD^\rb_{\cc}(\pOXS)$. We are then reduced to proving that there exists a natural transformation
\[
\Id_{\rD^\rb_{\rhol}(\DXS)}\To{\beta} \RH^S_X \circ \pSol_X
\]
such that, for any $\shm\in \rD^\rb_{\rhol}(\DXS)$, denoting by
\[
\beta_\shm:\shm\to \RH^S_X(\pSol_X(\shm))
\]
the unique morphism such that $\pSol_X(\beta_\shm) \circ \alpha_{\pSol_X(\shm)}=\Id_{\pSol_X(\shm)}$, we have
\[
\shm\xrightarrow[\sim]{~\textstyle\beta_\shm~}\RH^S_X(\pSol_X(\shm)).
\]
The proof of the existence of such a $\beta$ is reduced to that of the following result, which is equivalent to Theorem~\ref{RHH} by an argument already used in \cite[\S4.3]{MFCS2} (\cf introduction to Section \ref{S3}):

\begin{theo}\label{Tequivtf}
For any $\shm \in \rD^\rb_{\rhol}(\DXS)$
and for any $F\in\rD^\rb_\cc(\pOXS)$ the
complex $\Rhom_{\DXS}(\shm, \RH^S_X(F))$ belongs to $\rD^\rb_{\cc}(\pOXS)$.
\end{theo}

The proof of Theorem~\ref{Tequivtf} follows the ideas of \cite[Lem.\,4.1.4]{KS6} (see also Kashiwara's proof \cite[\S8.3]{Ka3} in the absolute case). The proof in the $S$-torsion case amounts to that of \loccit\ On the other hand, in order to prove Theorem~\ref{Tequivtf} in the general case, we apply induction on the dimension of the support and proceed by considering the case of $\DXS$-modules of D-type (normal crossing case).

We know (\cf\cite[\S3]{MFCS1}) that the functor $\pSol_X$ transforms duality on $\rD^\rb_\hol(\DXS)$ to Poincaré-Verdier duality on $\rD^\rb_{\cc}(\pOXS)$. A~consequence of the Riemann-Hilbert correspondence of Theorem \ref{RHH} and the full faithfulness of $\RH^S_X$ is the good behaviour of the functor $\RH^S_X$ with respect to Poincaré-Verdier duality on the one hand, and duality for $\DXS$-modules on the other hand.

\begin{coro}\label{D2}
For any $F\in\rD^\rb_{\cc}(\pOXS)$, there exists an isomorphism in $\rD^\rb_{\rhol}(\DXS)$
\[
\bD(\RH^S_X(F))\simeq \RH^S_X(\bD F)
\]
which is functorial in $F$.
\end{coro}

\subsubsection*{Acknowledgements}
We are grateful to Luca Prelli for useful advising in Section \ref{s5}. We thank the anonymous referee for valuable comments which helped us to improve the presentation of the article.

\section{Review on the relative holonomic \texorpdfstring{$\DXT$}{DXT}-modules and constructible complexes}\label{S1}

In this section, we review the main definitions and properties of the objects entering the relative Riemann-Hilbert correspondence. We refer to \cite{MFCS1, MFCS2, MFCS3} for details. We also give supplementary properties that will happen to be useful in the proof of the main results of this article.

\subsection{Holonomic \texorpdfstring{$\DXT$}{DXT}-modules}\label{subsec:prelimD}

We denote by $\DXT$ the subsheaf of $\shd_{\XT}$ of
relative differential operators with respect to the projection
\[
p_X:\XT\to T,
\]
that we simply denote by $p$ when there is no ambiguity. This is a Noetherian sheaf of rings. A coherent $\DXT$-module $\shm$ is said to be \emph{holonomic} if
\hbox{$\Char(\shm)\subseteq\Lambda\times T$} for some closed conic Lagrangian complex analytic subset~$\Lambda$ of $T^* X$ (see \cite[Lem.\,2.10]{FMF1} for a more precise description of the characteristic variety).

\begin{example}\label{exam:fsd}
Let $f_1,\dots,f_d:X\to\CC$ be holomorphic functions. We let here $T=\CC^d$ with coordinates $s_1,\dots,s_d$ and we consider the partially algebraic version $\shd_X[s_1,\dots,s_d]$ of $\DXT$. Let $M$ be a holonomic $\shd_X$-module and let $m$ be a local section of $M$. Extending \cite{Ka1} (which is the case $d=1$), one considers the $\shd_X[s_1,\dots,s_d]$-submodule $\shm=\shd_X[s_1,\dots,s_d]\cdot m\cdot f_1^{s_1}\cdots f_d^{s_d}$ of $M[(\prod_if_i)^{-1}][s_1,\dots,s_d]$. It is proved in \cite[Prop.\,13]{Maisonobe16} that $\shm$ is relatively holonomic.
\end{example}

\begin{example}\label{exam:mtm}
Set $S=\C^*$ with coordinate $z$. Any mixed twistor $\shd$-module, in the sense of \cite{Mochizuki11}, which consists of data parametrized by $\mathbb{P}^1$, gives rise, when restricting the parameter to $S=\CC^*$, to a holonomic $\DXS$-module (this is by definition, \cf \cite[Chap.\,1]{Sabbah05}).
\end{example}

We denote by
$\rD^\rb_{\hol}(\DXT)$ the full subcategory of $\rD^\rb_{\coh}(\DXT)$
whose complexes have holonomic cohomologies.

Given $t_o\in T$, let $i_{t_o}$ denote the inclusion $X\times \{t_o\}\hto \XT$. Following~\cite{MFCS1}, we denote by
\[
Li^{*}_{t_o}:\rD^\rb_\hol(\DXT)\to\rD^\rb_\hol(\shd_X)
\]
the derived functor
\begin{equation}\label{eq:Listar}
p_X^{-1}(\sho_T/\mathfrak{m}_{t_o})\overset{L}{\otimes}_{\pOXT}(\cbbullet),
\end{equation}
where $\mathfrak{m}_{t_o}$ is the maximal ideal of functions vanishing at $t_o$. Thanks to the variant of Nakayama's lemma \cite[Prop.\,1.9 \& Cor.\,1.10]{MFCS2}, the family of functors~$Li^{*}_{t_o}$ on $\rD^\rb_\hol(\DXT)$, for $t_o\in T$, is a conservative family, \ie if \hbox{$\phi:\shm\to \shn$} is a morphism in $\rD^\rb_\hol(\DXT)$ such that $Li^*_{t_o}\phi$ is an isomorphism in
$\rD^\rb_\hol({\mathcal D}_X)$ for each $t_o\in T$ then $\phi$ is an isomorphism (or, equivalently, using the mapping cone:
if $\shm\in\rD^\rb_\hol(\DXT)$ is such that $Li^*_{t_o}\shm=0$ for each $t_o\in T$ then $\shm=0$).

Recall (\cf\cite{MFCS1}) that a coherent $\DXT$-module is said to be \emph{strict} if it is $\pOXT$-flat. If $\shm$ is strict, $Li^*_{t_o}\shm$ consists of a single coherent $\shd_X$\nobreakdash-module $i^*_{t_o}\shm$ (in~degree zero). For example (recall that $\dim S=1$) a coherent $\DXS$-module $\shm$ is strict if and only if it has no $\pOXS$-torsion. If $\shm$ is possibly not strict, we shall denote by $t(\shm)$ its (coherent) submodule consisting of germs of sections which are torsion elements for the $\pOXS$-action, and $f(\shm):=\shm/t(\shm)$ is called its strict (or torsion-free) quotient. Therefore, the $\DXS$-module $\shm$ is strict if and only if $\shm\simeq f(\shm)$.

Given $\shm\in \rD^\rb_{\hol}(\DXT)$, the functor\enlargethispage{\baselineskip}%
\[
\bD(\shm):=\rh_{\DXT}(\shm, \DXT\otimes_{\sho_{\XT}}\Omega^{\otimes^{-1}}_{\XT/T})[d_X]
\]
provides a duality in $\rD^\rb_{\hol}(\DXT)$ but, contrary to the absolute case (\ie $\dim T=0$), this functor is not $t$-exact. For example, when the parameter space has dimension one, the lack of exactness of the dual functor on $\rD^\rb_{\hol}(\DXS)$ is due to the fact that the dual of a torsion holonomic $\DXS$-module $\shm$ is not concentrated in degree zero: if~$\shm\simeq\nobreak t(\shm)$ we have $\bD(\shm)\simeq \shh^1(\bD(\shm))[-1]$. On the other hand, if $\shn$ is a strict holonomic $\DXT$-module, then $\bD(\shn)\simeq \shh^0(\bD(\shn))$ and $\shh^0(\bD(\shn))$ is strict (see \cite[Prop.\,2]{MFCS2}), that is, $\shn$ is also dual holonomic: recall that a complex $\shn$ in $\rD^\rb_{\hol}(\DXT)$ is called \emph{dual holonomic} if it is in the heart of the $t$\nobreakdash-structure~$\Pi$ (see \cite[\S2]{FMF1}) which, by definition, is the $t$\nobreakdash-structure dual to the canonical $t$-structure.

We recall the following result in \cite[Lem.\,2.10]{FMF1}:\footnote{We keep the same numbering as in the published paper.}
\addtocounter{theorem}{2}
\begin{proposition}\label{Charhol}
For any
holonomic
$\shd_{X\times T/T}$-module $\shm$ we have
\[
\mathrm{Char}(\shm)=\bigcup_{i\in I} \Lambda_i\times T_i
\]
for some closed $\C^*$-conic irreducible Lagrangian subsets $\Lambda_i$ of $ T^*X$ and some
closed analytic subsets $T_i$ of $T$, and, locally on $X$, the set $I$ is finite.
Moreover $p_X(\mathrm{Supp}(\shm))=\bigcup_{i\in I} T_i$, hence it is an analytic subset of $T$, and
\[
\dim {\mathrm{Char}}(\shm)=\dim X+t, \quad \hbox{where } t=\dim p_X(\mathrm{Supp}(\shm))=
\sup_{i\in I} \dim T_i.
\]
\end{proposition}

\subsection{Relative constructible and perverse complexes}
\label{subsec:prelimtop}

\subsubsection{Relative local systems}
Following \cite{MFCS1} we say that a sheaf $\pOXT$-module $F$ is
$T$-locally constant coherent if, for each point $(x_0,t_o)\in \XT$ there exists a neighborhood $U=V_{x_0}\times T_{t_o}$ and
a coherent sheaf $G^{(x_0,t_o)}$ of $\sho_{T_{t_o}}$-modules such that $F_{|U}\cong p_{V_{x_0}}^{-1}(G^{(x_0,t_o)})$. We refer to \cite[App.]{MFCS2} for basic properties.

By definition, $\rD^\rb_{\lc\; \coh}(\pOXT)$ is the full subcategory of $\rD^\rb(\pOXT)$ whose complexes have $T$-locally constant coherent cohomologies (notice that, for such an $F$, $F_{|\{x_0\}\times T}\in \rD^\rb_\coh(\sho_T)$). We refer to \cite[\S2]{MFCS1} for more properties.

\subsubsection{Relative \texorpdfstring{$\R$}{R}-constructibility}\label{rc}
In the following, we will have to consider derived categories $\rD^\star$ with $\star=\rb$ or $\star={}-$. We denote by $\rD^\star_\rc(\pOXT)$ the full subcategory of $\rD^\star(\pOXT)$ whose objects~$F$ admit a $\mu$-stratification $(X_\alpha)$ of $X$ such that $i_{\alpha}^{-1}(F)\in\rD^\star_{\lc\; \coh}(p_{X_\alpha}^{-1}\sho_T)$ for any~$\alpha$. We refer to \cite[\S2]{MFCS1} for details.

Objects of $\rD^-_\rc(\pOXT)$ can be given a simple representative. Let us denote by $\shs$ the full additive subcategory of $\Mod_{\rc}(\pOXT)$ whose objects are sheaves which can be expressed as locally finite direct sums of terms of the form $\C_{\Omega}\boxtimes \sho_V:=\pOXT\otimes \C_{\Omega\times V}$ for some relatively compact open subanalytic subsets $\Omega$ in $X$ and $V$ in~$T$. A morphism $\varphi:\C_{\Omega}\boxtimes \sho_V\to\C_{\Omega'}\boxtimes \sho_{V'}$ is easily described: setting $\Omega''=\Omega\cap\Omega'$ and $V''=V\cap V'$, $\varphi$ is the extension by zero of its restriction $\varphi_{|\Omega''\times V''}:\pOXT{}_{|\Omega''\times V''}\to\pOXT{}_{|\Omega''\times V''}$, which is the multiplication by a section of $\pOXT$ on each connected component of $\Omega''\times V''$.

Following the terminology of \hbox{\cite[App.\,A]{KS4}}, we say that an object $F$ of $\Mod(\pOXT)$ is $\shs$-coherent if there exist $L\in\shs$ and an epimorphism $L\to F$, and if, for any morphism $L'\to F$ with $L'$ in $\shs$, there exist~$L''$ in $\shs$ and a morphism $L''\to L'$ such that $L''\to L'\to F$ is exact. It follows from \cite[Prop.\,3.5]{MFCS2} that the category of $\shs$\nobreakdash-coherent objects of $\Mod(\pOXT)$ is equal to $\Mod_{\rc}(\pOXT)$ and the category $\rD^-_{\shs\text{-}\coh}(\Mod(\pOXT))$ is nothing but $\rD^-_{\rc}(\pOXT)$. On the other hand, one defines the category $\rD^-_\coh(\shs)$ as in \cite[p.\,63]{KS4}, with the identification $\mathbf{A}=\Mod(\pOXT)$ and $\mathbf{P}=\shs$ (the functor~$L$ of \loccit\ is here the inclusion, $H$ is the restriction of $\Hom_{\Mod(\pOXT)}$ to $\shs\times\Mod(\pOXT)$ and $\alpha$ is the identity). In the present situation, we have $\rD^-_\coh(\shs)=\rD^-(\shs)$.

\begin{proposition}[\cf {\cite[Th.\,A.5]{KS4}}]\label{P:AF}
The natural functor
\[
L:\rD^-(\shs)\to\rD^-_{\rc}(\pOXT)
\]
is an equivalence of categories.
\end{proposition}

\begin{proof}
We only need to check that the conditions for applying \cite[Th.\,A.5]{KS4} are fulfilled, that is, that the pair $(\Mod(\pOXT),\shs)$ satisfies the properties (A.1)--(A.4) for $(\mathbf{A},\mathbf{P})$ in \loccit, and only (A.3) is not obvious. It is proved in the lemma below.
\end{proof}

\begin{lemma}\label{PTMF6}
Let $F$ and $G$ be objects of $\Mod(\pOXT)$, let $\psi: F\to G$ be an epimorphism, let $\sht\in\shs$ and let $g:\sht\to G$ be given. Then there exist an object $\sht'\in\shs$, an epimorphism $\psi': \sht'\to \sht$ and a morphism $g': \sht'\to F$ such that $\psi g'=g\psi'$.
\end{lemma}

\begin{proof}
By $\pOXT$-linearity we may reduce to the case $\sht={\pOXT}\otimes \C_{\Omega\times V}$, for some relatively compact open subanalytic subsets $\Omega,V$ respectively in $X$ and~$T$.

Let $e$ be the section $1\in\Gamma(\Omega\times V,\sht)$. By the assumption on $\psi$, we can cover~$\Omega$ (\resp $V$) by a locally finite family of relatively compact open subanalytic sets $\Omega_i\subset\nobreak X$ (\resp $V_i\subset T$), $i=1,\dots, m$, and find sections $f_i\in\Gamma(\Omega_i\times V_i,F)$ such that $\psi|_{\Omega_i\times V_i}(f_i)=g(e)|_{\Omega_i\times V_i}$. The morphism $\CC_{|\Omega_i\times V_i}\to F_{|\Omega_i\times V_i}$ extends in a unique way as a morphism $\CC_{\Omega_i\times V_i}\to F$ and, by $\pOXT$-linearity, as a morphism $\pOXT\otimes \C_{\Omega_i\times V_i}\to F$.

Setting $\sht':=\bigoplus_{i=1}^m\pOXT\otimes \C_{\Omega_i\times V_i}$, we obtain in this way a $\pOXT$-linear morphism $g':\sht'\to F$. On the other hand, by the covering property, the natural morphism $\CC_{\Omega_i\times V_i}\to\CC_{\Omega\times V}$ which extends $\id:\CC_{\Omega_i\times V_i|\Omega_i\times V_i}\to\CC_{\Omega\times V|\Omega_i\times V_i}$ induces an epimorphism $\bigoplus_{i=1}^m\C_{\Omega_i\times V_i}\to\CC_{\Omega\times V}$ and, by $\pOXT$-linearity, an epimorphism $\psi':\sht'\to\sht$, which clearly satisfies $\psi g'=g\psi'$.
\end{proof}

We note the following, to be used in the course of the proof of Lemma~\ref{LdirS}:
\begin{remark}\label{C:functors}
Let $\Phi,\Psi$ be two triangulated functors from $\rD^-_\rc(\pOXT)$ to a triangulated category $\mathsf{C}$.
Any morphism of functors $\eta_\shs:\Phi_\shs\to\Psi_\shs$  (with
$\Phi_\shs=\Phi\circ L$, $\Psi_\shs=\Psi\circ L$) can be extended to a morphism of functors $\eta:\Phi\to\Psi$.
\end{remark}

\subsubsection{Relative \texorpdfstring{$\C$}{C}-constructibility}\label{subsubsec:relativeCc}
By definition (\cf \cite[Def.\,2.19]{MFCS1}), the full subcategory $\rD^\rb_\cc(\pOXT)$ consists of objects of $\rD^\rb_{\rc}(\pOXT)$ whose microsupport is $\C^*$-conic. We call these objects \emph{$T$-$\C$-constructible complexes}.

For any $t_o\in T$, there is a functor
\[
Li^{*}_{t_o}:\rD^\rb(\pOXT)\to\rD^\rb(\C_X)
\]
also defined by \eqref{eq:Listar}. It sends $\rD^\rb_\rc(\pOXT)$ to $\rD^\rb_\rc(\C_X)$ and $\rD^\rb_\cc(\pOXT)$ to $\rD^\rb_\cc(\C_X)$.

Recall (\cf \cite[Prop.\,2.2]{MFCS1}) that a variant of Nakayama's lemma holds for complexes $F$ in $\rD^\rb(\pOXT)$
whose cohomology objects $\shh^jF$ have fibers $\shh^jF_{(x,s)}$ of finite type over $\sho_{T,s}$ for any $(x,s)\in \XT$. As a consequence, the family of functors
$(Li^{*}_{t_o})_{{t_o}\in T}$ on $\rD^\rb_\rc(\pOXT)$ (\resp $\rD^\rb_\cc(\pOXT)$) is a conservative family (in particular,
if \hbox{$F\!\in\!\rD^\rb_\cc(\pOXT)$} satisfies $Li^*_{t_o}F=0$ for each ${t_o}\in T$, then $F=0$).

\subsubsection{Perversity}\label{subsubsec:perv}
The category $\rD^\rb_\cc(\pOXT)$ is endowed with a perverse $t$\nobreakdash-structure defined in \cite[\S2.7]{MFCS1}
as the relative analogue to the middle perverse $t$-structure in the absolute case where $\dim T=0$:
\begin{itemize}
\item
$\pD^{\leq 0}_{\cc}(\pOXT)$ is the full subcategory of objects $F$ of $\rD^\rb_{\cc}(\pOXT)$ such that there exists an adapted $\mu$-stratification $(X_\alpha)$ of $X$ for which $i_x^{-1}F\in {\rD}_\coh^{\leq -d _{X_{\alpha}}}(\sho_T)$ for any $x\in X_\alpha$ and any $\alpha$.

\item
$\pD^{\geq 0}_{\cc}(\pOXT)$ is the full subcategory of objects $F$ of $\rD^\rb_{\cc}(\pOXT)$ such that there exists an adapted $\mu$-stratification $(X_\alpha)$ of $X$ for which $i_x^{!}F\in {\rD}_\coh^{\geq d _{X_{\alpha}}}(\sho_T)$ for any $x\in X_\alpha$ and any $\alpha$.
\end{itemize}

The heart of this $t$-structure is the abelian category of \emph{relative perverse sheaves} denoted by $\perv(\pOXT)$. We often omit the word ``relative''.

In analogy with the $\DXS$-module counterpart ($S$ is a curve), following
\cite[Prop.\,3.12]{FMF1}, we say that a perverse sheaf is \emph{torsion} if it belongs to
the subcategory $\perv(\pOXS)_\rt$ of $\perv(\pOXS)$ whose objects $F$ satisfy $\codim p(\supp F)\geq 1$ (\cf \cite[Cor.\,3.1]{FMF1} for this condition), while a perverse sheaf is called {\emph{strictly perverse} if it belongs to the full subcategory $\perv(\pOXS)_{\tf}$ of $\perv(\pOXS)$ whose objects~$F$ satisfy
$Li^*_sF\in\perv(\C_X)$ for all $s\in S$. The category $\perv(\pOXS)_\rt$ is a full thick abelian subcategory of the category $\perv(\pOXS)$.

We denote by $\rD^\rb_{\cc}(\pOXS)_{\rt}$ the thick subcategory of
$\rD^\rb_{\cc}(\pOXS)$ whose objects have support in $X\times T$, where $T$ is a subset of $S$ with $\dim T=0$ or, equivalently, whose perverse cohomologies belong to
$\perv(\pOXS)_\rt$.

Given an object $F$ of $\rD^\rb_{\cc}(\pOXS)$, the functor $\bD$ defined by
\[
\bD(F)=\rh_{\pOXS}(F,\pOXS)[2d_X]
\]
provides a duality in $\rD^\rb_{\cc}(\pOXS)$,
which is however not $t$-exact with respect to the perverse $t$-structure.
For example, if $F$ is a torsion perverse sheaf, then $\bD(F)\simeq {}^p\shh^1(\bD(F))[-1]$
(it is a perverse sheaf shifted in degree $1$),
while if~$F$ is a strictly perverse sheaf, $\bD(F)\simeq {}^p\shh^0(\bD(F))$ is perverse too.

Let us recall that an object $F$ of $\rD^\rb_{\cc}(\pOXS)$ is called
\emph{dual perverse} if it is in the heart of the $t$-structure $\pi$,
which by definition is the $t$-structure dual
to the perverse $t$-structure introduced in \cite[\S2.7]{MFCS1}.
By \cite[Lem.\,1.4]{MFCS2}, a complex $F\in \rD^\rb_{\cc}(\pOXS)$
is perverse and dual perverse if and only if it is strictly perverse.

\subsection{The relative solution functor}\label{subsec:solfunctor}
The solution functor for a coherent $\DXT$-module or an object of $\rD^\rb_\coh(\DXT)$ is defined by
\begin{align*}
\pSol_X: \rD^{\rb}_{\coh}(\DXT)&\to \rD^\rb(\pOXT) \\
\shm &\mto \rh_{\DXT}(\shm,\sho_{\XT})[d_X]
\end{align*}
By \cite[Th. 3.7]{MFCS1}, when restricted to $\rD^{\rb}_{\hol}(\DXT)$ the solution functor $\pSol_X$ takes values in $\rD^\rb_{\cc}(\pOXT)$
and by \cite[Cor.\,4.3]{FMF1},}
is $t$-exact with respect to the $t$-structure $\Pi$ in $ \rD^{\rb}_{\hol}(\DXT)$ and the perverse one
$p$ in $\rD^\rb_{\cc}(\pOXT)$.

\subsection{Regular holonomic complexes of \texorpdfstring{$\DXS$}{DXS}-modules}\label{S1d}
In this section, we review the notion of relative regularity as introduced in \cite[\S2.1]{MFCS2} and recall the fundamental example of relative $\DXS$-modules of D-type.

\begin{definition}[Regularity, {\cite[Def.\,2.1]{MFCS2}}]
A~holonomic $\DXS$-module $\shm$ is said to be \emph{regular} if, for any $s_o\in S$, the object $Li^*_{s_o}\shm$ of $\rD^\rb_\hol(\shd_X)$ has regular holonomic cohomologies.
\end{definition}

\begin{example}\mbox{}
\begin{enumerate}

\item{In Example \ref{exam:fsd} let us assume that $M$ is regular. Then the $\DXS$-module generated by $\shm$ is regular.}
\item{In Example \ref{exam:mtm}, assume that the mixed twistor $\shd$-module is regular in the sense of \cite[Def.\,4.1.2]{Sabbah05}. Then the underlying holonomic $\DXS$-module $\shm$ is regular.}
\end{enumerate}
\end{example}

According to \cite[\S2.1]{MFCS2}, we say that an object $\shm\in\rD^\rb_{\hol}(\DXS)$ is \emph{regular} if each of its cohomology modules is regular.

\begin{remark}
An object $\shm$ of $\rD^\rb_{\hol}(\DXS)$ is regular if and only if, for each $s_o\in S$, the object $Li^*_{s_o}\shm$ of $\rD^\rb_\hol(\shd_X)$ has regular holonomic cohomology. Indeed, we argue by induction on the amplitude of the complex~$\shm$.
Without loss of generality, we may assume that $\shm\in \rD^{\geq 0}_{\hol}(\DXS)$ and we consider the following distinguished triangle
\[
\shh^0\shm\to\shm\to \tau^{\geq1}\shm\To{+1}
\]
(where $\tau^{\geq 1}$ is the truncation functor with respect to the natural $t$-structure on $\rD^\rb_{\hol}(\DXS)$). We deduce $\shh^{-1}Li^*_{s_o}\shh^0(\shm)\simeq \shh^{-1}Li^*_{s_o}(\shm)$ and an exact sequence
\[
0\to \shh^0 Li^*_{s_o}\shh^0\shm\to
\shh^0 Li^*_{s_o}\shm\to \shh^0 Li^*_{s_o}\tau^{\geq 1}\shm\to 0.
\]
(Note that $\shh^kLi^*_{s_o}\shh^0(\shm)=0$ for $k\neq0, -1$.) The assertion follows from the induction hypothesis and the property that the category of regular holonomic $\DX$-modules is closed under sub-quotients in the category $\Mod_{\coh}(\DX)$.
\end{remark}

We recall the following result in \cite{MFCS2}:

\begin{proposition}[\cf {\cite[Cor.\,2.4]{MFCS2}}]\label{Dirim}
Let $f: Y\to X$ be a proper morphism of complex manifolds. Then, for each $\shm\in\rD^\rb_{\rhol}(\DYS)$ whose cohomology is $f$-good, the pushforward $\Df_*\shm$ is an object of $\rD^\rb_{\rhol}(\DXS)$.
\end{proposition}

\subsubsection{Proof of Theorem \ref{inverseimage} for a smooth morphism}\label{subsubsec:smoothpullbackreghol}
Let $f: Y\to X$ be a smooth morphism and let $\shm$ be an object of $\rD^\rb_\rhol(\DXS)$. Due to the locality of the regular holonomic property, we may assume that $f$ is a projection $Y=Z\times X\to X$. In that case, $\shd_{(Y\to X)\times S}$ is $f^{-1}\DXS$-flat, so $\shh^j\Df^*\shm\simeq\Df^*\shh^j\shm$ for every $j$, and we can assume that $\shm=\shh^0\shm$. As in the absolute case one checks that $\Char(\Df^*\shm)=T^*_ZZ\times\Char(\shm)$, so $\Df^*\shm
\in \rD^\rb_\hol(\DYS)$. Moreover, the commutativity of
\[
\xymatrix{
Y
\ar[r]^-{i_s} \ar[d]_{f}
&
\YS \ar[d]^{f} \\
X \ar[r]^-{i_s} & X\times S \\
}
\]
implies that $Li^*_s\Df^*(\shm)\simeq \Df^*Li_s^*\shm$, and the latter is known to be regular holonomic on $Y$. Therefore, Theorem \ref{inverseimage} is proved for $f$ smooth.\qed

\subsubsection{\texorpdfstring{$\DXS$}{DXS}-modules of D-type}
They are the fundamental examples of regular holonomic $\DXS$-modules, so we recall their definition. Let $D$ be a normal crossing divisor in $X$ and let $j:X^*:=X\moins D\hto X$ denote the inclusion
(we will also denote by $j$ the morphism $j\times {\id}_S$).
Let~$F$ be a coherent $S$\nobreakdash-locally constant sheaf on $\XsS$ and let $(V,\nabla)=(\sho_{\XsS}\otimes_{\pOXS}\nobreak F,\rd_{X/S})$ be the associated coherent $\sho_{\XsS}$-module with flat relative connection. There exists a coherent $\sho_S$-module~$\cG$ such that, if $U$ is any contractible open set of $X^*$, then $F_{|\US}\simeq p_U^{-1}\cG$.

Let $\varpi:\wt X\to X$ denote the real oriented blowing up of $X$ along the components of~$D$. Denote by $\wtj:X^*\hto\wt X$ the inclusion, so that $j=\varpi\circ\wtj$. Let $x^o\in D$, $\wt x^o\in\varpi^{-1}(x^o)$ and let $s^o\in S$. Choose local coordinates $(x_1,\dots,x_n)$ at $x^o$ such that $D=\{x_1\cdots x_\ell=0\}$ and consider the associated polar coordinates $(\rhog,\thetag,\boldsymbol{x}'):=(\rho_1,\theta_1,\dots,\rho_\ell,\theta_\ell,x_{\ell+1},\dots,x_n)$ so that $\wt x^o$ has coordinates $\rhog^o=\nobreak0$, $\thetag^o$, $\boldsymbol{x}^{\prime o}=0$.

A local section $\wt v$ of $(\wtj_*V)_{(\wt x^o,s^o)}$ is said to have \emph{moderate growth} if for some system of generator of $\cG_{s^o}$, and some neighbourhood
\[
U_\epsilon:=\{\Vert\rhog\Vert<\epsilon,\Vert\boldsymbol{x}'\Vert<\epsilon,\Vert\thetag-\thetag^o\Vert<\epsilon\}
\]
($\epsilon$ small enough) on which it is defined, its coefficients on the chosen generators of~$\cG_{s^o}$ (these are sections of $\sho(U^*_\epsilon\times U(s^o))$ for a small enough neighbourhood~$U(s^o)$ of $s^o$ in $S$, and $U^*_\epsilon:=U_\epsilon\moins\{\rho_1\cdots\rho_\ell=0\}$) are bounded by $C\rhog^{-N}$, for some $C,N>0$.
A local section $v$ of $(j_*V)_{(x^o,s^o)}$ is said to have \emph{moderate growth}
if for each~$\wt x^o$ in $\varpi^{-1}(x^o)$, the corresponding germ in $(\wtj_*V)_{(\wt x^o,s^o)}$ has moderate growth.

On the other hand (\cf \cite[Def.\,2.10]{MFCS2}), a coherent $\DXS$-module $\shl$ is said to be of D-type with singularities along a normal crossing divisor $D\subset X$ if it satisfies the following conditions:
\begin{enumerate}
\item $\Char(\shl)\subset(\pi^{-1}(D)\times S)\cup (T^*_XX\times S)$,
\item $\shl$ is regular holonomic and strict,
\item $\shl\simeq \shl(*(D\times S))$.
\end{enumerate}

The following result is proved in \cite[Th.\,2.6, Cor.\,2.8 \& Prop.\,2.11]{MFCS2}:

\begin{theorem}\mbox{}\label{th:Dtype}
\begin{enumerate}
\item\label{th:Dtype1}
The subsheaf $\wt V$ of $j_*V$ consisting of local sections having moderate growth is stable by $\nabla$ and it is $\sho_{\XS}(*D)$-coherent. Moreover, $\wt V$ is a regular holonomic
$\DXS$-module with characteristic variety contained in $\Lambda\times S$, where $\Lambda$ is the union of the conormal spaces of the natural stratification of $(X,D)$.
\item\label{th:Dtype2}
A coherent $\DXS$-module $\shl$ is of D-type on $(X,D)$ if and only if it is isomorphic to some $\wt V$ as in \eqref{th:Dtype1}.
\end{enumerate}
\end{theorem}

\section{The relative Riemann-Hilbert functor \texorpdfstring{$\RH^S$}{RHS}}\label{S2}

In this section we recall the definition of the relative Riemann-Hilbert functor $\RH^S(\cbbullet)$ introduced in \cite{MFCS2} and state some supplementary results needed in the sequel.

\subsection{Relative subanalytic sites and relative subanalytic sheaves}\label{subsec:relsubanalytic}
For details on this subject we refer to \cite{TL} and \cite{EP}. We also refer to \cite{KS5} as a foundational paper and to \cite{KS3} for a detailed exposition on the general theory of sheaves on sites.

Let $X$ and $T$ be real or complex analytic manifolds. One denotes by $\Op(\XT)$ the family of open subsets of $\XT,$ by $\Op((\XT)_{\sa})\subset \Op (\XT)$ the family of open subanalytic sets; $\sht:=\Op^\mathrm{c}((\XT)_{\sa})$ denotes the family of relatively compact open subanalytic subsets of $\XT$ and $\sht'\subset \sht$ denotes the family of finite unions of relatively compact open subanalytic sets of the form $U\times V$.

The product $\XT$ is both a $\sht$- as well as a $\sht'$-space.
The associated sites $({\XT})_{\sht}$ and $({\XT})_{\sht'}$ are, respectively, the subanalytic site $(\XT)_{\sa}$, for which the coverings of an element $\Omega\in \Op((\XT)_{\sa})$ are the locally finite coverings with elements in $\sht$, and the site denoted by $X_{\sa}\times T_{\sa}$, for which the coverings of $\Omega\in\sht'$ are the coverings with elements in $\sht'$ which admit a finite subcovering.

We shall denote by $\rho_{T}$ the natural functor of sites $\rho_T:\XT \to (\XT)_{\sa}$ associated to the inclusion $\Op_{\sa}(\XT) \subset \Op(\XT)$.
Accordingly, we shall consider the associated functors $\rho_{T*}, \rho_T^{-1}, \rho_{T!}$.

We shall also denote by $\rho'_T:\XT \to X_{\sa}\times T_{\sa}$ the functor of sites associated to the inclusion $\sht'\subset \Op (\XT)$. Following \cite{KS3} we have functors~$\rho'_{T*}$ and $\rho'_{T !}$ from $\Mod(\CC_{\XT})$ to $\Mod(\CC_{X_{\sa}\times T_{\sa}})$.
We simply denote by~$\rho$, \resp $\rho'$, the previous morphism when there is no ambiguity.

Subanalytic sheaves are defined on the subanalytic site of a real analytic manifold, and relative subanalytic sheaves are defined on the subanalytic site $X_{\sa}\times T_{\sa}$. We refer to \cite{TL}
for the detailed construction of the relative subanalytic sheaves $\Db^{t,T}_{\XT}$ (where $X$ and $T$ are real analytic) and $\sho^{t,T}_{\XT}$ in the complex framework (denoted $\Db^{t,T,\sharp}_{\XT}$ and 
 $\sho^{t,T,\sharp}_{\XT}$ in \cite{TL}).

They are both $\rho'_!\shd_{\XT}$-modules (either in the real or the complex case) as well as a $\rho'_*\pOXT$-modules when $T$ is complex.

If $\Db^t_{\XT}$ denotes the subanalytic sheaf of tempered distributions introduced by Kashiwara-Schapira in \cite{KS5}, we have, for $U\in\Op(X_{\sa})$ and $V\in\Op(T_{\sa})$
\begin{align*}
\Gamma(U\times V; \Db^{t,T}_{\XT})&=\varprojlim_{W\Subset V}\Gamma(U\times W; \Db^t_{\XT})\\
&\simeq \Gamma (X\times V; \rho'^{-1} \Gamma_{U\times T}\Db^t_{\XT})\\
&\simeq \Gamma(X\times V; \tho(\C_{U\times T}, \Db_{\XT})).
\end{align*}

Moreover, when~$X$ is also complex, considering the complex conjugate structure $\overline{X}$ on $X$ (\resp $\overline{T}$ on~$T$) and the underlying real analytic \hbox{structure}~$X_{\R}$ (\resp $T_{\R}$),
we have
\[
\sho^{t,T}_{\XT}=\Rhom_{\rho'_!\shd_{\overline{X}\times\overline{T}}}(\rho'_!\sho_{\overline{X}\times\overline{T}},\shd^{t,T}_{\XT}),
\]
where we omit the reference to the real structures.

\subsection{The functors \texorpdfstring{$\TH^S$}{THS} and \texorpdfstring{$\RH^S$}{RHS}}\label{subsec:RHS}

\subsubsection{The functor \texorpdfstring{$\TH^S$}{THS}}
When $X$ is a real analytic manifold and $S$ is a complex curve,  we define the triangulated functor $$\TH^S_X: \rD^\rb_\rc(\pOXS)^\mathrm{op}\to\rD^+(\shd_{X\times S_{\R}/S})$$ given by  
$$F \mto \TH^S_X(F):=
\rho'^{-1}\rh_{\rho'_*\pOXS}(\rho'_*F, \Db^{t,S}_{\XS})$$
where $\shd_{X\times S_{\R}/S}$ denotes the sheaf of linear differential operators with real analytic coefficients on $X\times S_{\R}$ which commute with $\pOXS$.

Recall that, as a consequence of \cite[Prop. 4.7]{TL} we have \begin{align*}
\TH^S_X(\pOXS\otimes \C_{H\times V})&\simeq \Rhom(\C_{X\times V}, \tho(\C_{H\times S}, \Db_{\XS}))\\
&\simeq R\Gamma_{X\times V}\tho(\C_{H\times S}, \Db_{\XS})\quad(\text{\cite[(2.6.9)]{KS1}})
\end{align*}
for any relatively compact locally closed, \resp open, subanalytic subsets $H$ of~$X$, \resp $V$ of $S$. If $H=Z$ is closed, we have $\tho(\C_{Z\times S}, \Db_{\XS})=\Gamma_{Z\times S}\Db_{\XS}$ by definition. We conclude:
\begin{equation}\label{ETMF1}
\TH^S_X(\pOXS\otimes \C_{Z\times V})\simeq \Gamma_{Z\times V} \Db_{\XS}.
\end{equation}
On the other hand, if $H=\Omega$ is open, since $\tho(\C_{\Omega\times S}, \Db_{\XS})$ is a c-soft sheaf, we obtain
\begin{equation}\label{ETMF2}
\TH^S_X(\pOXS\otimes \C_{\Omega\times V})\simeq \Gamma_{X\times V} \tho(\C_{\Omega\times S}, \Db_{\XS}).
\end{equation}

\subsubsection{The functor \texorpdfstring{$\RH^S$}{RHS}}\label{subsubsec:real}
If $X$ is a complex manifold and $S$ is a complex curve,
$\RH^S_X:\rD^\rb_\rc(\pOXS)^\mathrm{op}\to\rD^\rb(\DXS)$ is given by the assignment
\begin{equation}
\begin{split}\label{RHS1}
F\mto \RH^S_X(F)&:=
\rho'^{-1}\rh_{\rho'_*\pOXS}(\rho'_*F, \sho^{t,S}_{\XS})[d_X]\\
&\phantom{:}\simeq \Rhom_{\shd_{\overline{X}\times\overline{S}}}(\sho_{\overline{X}\times\overline{S}}, \TH^S_X(F))[d_X],
\end{split}
\end{equation}
the last isomorphism being called here ``realification procedure'' for short (\cf\cite[(3.16)]{MFCS2}).

We collect below some results in \cite{MFCS2} which will be useful in the sequel.
The first gives the behaviour of $Li^*_s$ with respect to $\RH^S$.

\begin{proposition}[\cf{\cite[Prop.\,3.29]{MFCS2}}]\label{com}
For each $s\in S$ there is an isomorphism of functors on $\rD^\rb_{\rc}(p_X^{-1}\sho_S)$
\[
Li^*_s\RH^S_X[-d_X](\cbbullet)\simeq \tho(Li^*_s(\cbbullet), \sho_X),
\]
where $X$ is identified to $X\times\{s\}$ and $X_{\sa}$ is identified to $X_{\sa}\times \{s\}$.
\end{proposition}

\begin{theorem}[\cf {\cite[Th.\,3]{MFCS2}}]\label{TRHS}
Let $F\in\rD^\rb_{\cc}(p_X^{-1}\sho_S)$. Then $\RH^S_X(F)\in\rD^\rb_{\rhol}(\DXS)$ and we have an isomorphism $F\simeq \pSol_X(\RH^S_X(F))$ which is functorial in $\rD^\rb_{\cc}(p_X^{-1}\sho_S)$.
\end{theorem}

Let us fix a normal crossing divisor $D$.

\begin{theorem}[\cf {\cite[Prop.\,2.11 \& Lem.\,4.2]{MFCS2}}]\label{Dt}
The category of holonomic $\DXS$-modules $\shl$ of D-type with singularities along $D$ is equivalent to
the category of locally free $p_{X^*}^{-1}\sho_S$-modules with $X^*:=X\setminus D$
under the correspondence
\begin{equation*}
\shl\mto \shh^0\DR_X(\shl)_{|X^*\times S}, \qquad F\mto \RH^S_X(j_! \bD(F[d_X]))=\shl.
\end{equation*}
\end{theorem}

\subsection{Some functorial properties}
Let $f:Y\to X$ be a morphism of real or complex analytic manifolds. We denote similarly the morphism $f\times\id:\YS\to X\times S$. We consider in this section the corresponding pullback functor.

\subsubsection{Pullback with respect to \texorpdfstring{$X$}{X} for \texorpdfstring{$\TH^S$}{THS}}

\begin{proposition}\label{Ldir}
For any morphism $f:Y\to X$ of real analytic manifolds there exists a morphism of functors from $\rD^-_\rc(\pOXS)^\mathrm{op}$ to $\rD^+(\shd_{X\times S_{\R}/S})$:
\begin{equation}\label{eq:Ldir}
\Df_!\TH^S_Y(f^{-1}\cbbullet)\to \TH^S_X(\cbbullet).
\end{equation}
\end{proposition}

\begin{proof}
We first define the desired morphism functorially on the category $\shs$ introduced in Section \ref{rc}. We deduce from \eqref{ETMF2} that the objects of~$\shs$ are acyclic for $\TH^S_X(\cbbullet)$ and for $\TH^S_Y(f^{-1}\cbbullet)$.

\begin{lemma}\label{LdirS}
For any morphism $f:Y\to X$ of real analytic manifolds there exist functors 
\begin{align*}
\Df_!\TH^S_Y(f^{-1}\cbbullet)_{|\shs^\mathrm{op}}:\shs^\mathrm{op} &\to \mathsf{C}^+(\shd_{X\times S_{\R}/S}) \\
\TH^S_X(\cbbullet)_{|\shs^\mathrm{op}}:\shs^\mathrm{op} &\to  \mathsf{C}^+(\shd_{X\times S_{\R}/S}).
\end{align*}
and there exists a canonical morphism of functors
\[
\Df_!\TH^S_Y(f^{-1}\cbbullet)_{|\shs^\mathrm{op}}\to \TH^S_X(\cbbullet)_{|\shs^\mathrm{op}}
\]
inducing the morphism \eqref{eq:Ldir}.
\end{lemma}

\begin{proof}
Let $\Omega$, \resp $V$, be a relatively compact open subanalytic set in $X$, \resp $S$ (here, we \hbox{consider} the real analytic structures in $X$ and $S$, also usually identified by~$-_{\R}$). We set $Z=X\moins\Omega$ and we start by considering the sheaf $G=\pOXS\otimes\C_{Z\times V}$. According to \eqref{ETMF1}, we have $\TH^S_X(G)=\Gamma_{Z\times V}\Db_{\XS}$, \hbox{regarded} as a $\shd_{X_{\R}\times S_{\R}/S}$-module.

On the other hand, the integration of distributions induces a morphism
\[
\int_f:f_!(\Db_{\YS}\otimes_{\sha_{\YS}} \omega_{\YS/S})\to \Db_{\XS}\otimes_{\sha_{\XS}} \omega_{X\times S/S}.
\]
One then mimics \cite[Prop.\,4.3]{KS4} by replacing the transfer module $\shd_{\YS\to X\times S}$ for the morphism $\YS\to X\times S$ in the absolute sense by the relative one $\shd_{(Y\to X)\times S/S}:=\sha_{\YS}\otimes_{f^{-1}\sha_{\XS}} f^{-1}\shd_{X\times S_{\R}/S}$. Let us set $\Db_{\YS}^\vee=\Db_{\YS}\otimes_{\sha_{\YS}} \omega_{\YS/S}$ and let $\mathrm{Sp}_{\sbullet}(\shd_{(Y\to X)\times S/S})$ denote the Spencer resolution of $\shd_{(Y\to X)\times S/S}$ (we recall that, for $k\in\NN$, $\mathrm{Sp}_{k}(\shd_{(Y\to X)\times S/S})$ are locally free over $\shd_{\YS_{\R}/S}$). The terms of the complex in $\mathsf{C}^+(\shd_{\YS_{\R}/S})$
\begin{align*}
\shc_{f^{-1}Z\times V}&:=(\Gamma_{f^{-1}Z\times V}\Db_{\YS}^\vee)\otimes_{\shd_{\YS_{\R}/S}}\mathrm{Sp}_{\sbullet}(\shd_{(Y\to X)\times S/S})\\
&\hphantom{:}=\Gamma_{f^{-1}Z\times V}\bigl(\Db_{\YS}^\vee\otimes_{\shd_{\YS_{\R}/S}}\mathrm{Sp}_{\sbullet}(\shd_{(Y\to X)\times S/S})\bigr)
\end{align*}
are thus c-soft sheaves. Hence, the object 
\hbox{$\Df_!\TH^S(\pOYS\otimes\C_{f^{-1}Z\times V})$} is represented by the complex
\[
f_!\shc_{f^{-1}Z\times V}\simeq\Gamma_{Z\times V}f_!\bigl(\Db_{\YS}^\vee\otimes_{\shd_{\YS_{\R}/S}}\mathrm{Sp}_{\sbullet}(\shd_{(Y\to X)\times S/S})\bigr)
\]
in $\mathsf{C}^+(\shd_{X\times S_{\R}/S})$. As in \loccit, we get a morphism in $\mathsf{C}^+(\shd_{X\times S_{\R}/S})$:
\[
\varphi_{Z\times V}:f_!\shc_{f^{-1}Z\times V}\to \Gamma_{Z\times V}\Db_{\XS}=\TH^S_X(G).
\]

We now consider the object $L=\pOXS\otimes\C_{\Omega\times V}$ of $\shs$. From the short exact sequence
\[
0\to \C_{\Omega\times V}\to \C_{X\times V}\to \C_{Z\times V}\to 0,
\]
we \emph{define} $f_!\shc_{f^{-1}(\Omega\times V)}$ as the cokernel of the natural morphism $f_!\shc_{f^{-1}Z\times V}\to f_!\shc_{f^{-1}X\times V}$ in $\mathsf{C}^+(\shd_{X\times S_{\R}/S})$.

Therefore, on the one hand, the complex $f_!\shc_{f^{-1}(\Omega\times V)}$ is a representative of \hbox{$\Df_!\TH^S_Y(f^{-1}L)$} and $\TH^S_X(L)$ is the cokernel in $\Mod(\shd_{X\times S_{\R}/S})$ of
\[
\TH^S_X(\pOXS\otimes\C_{Z\times V})\hto\TH^S_X(\pOXS\otimes\C_{X\times V}).
\]
On the other hand, by completing the commutative diagram
\[
\xymatrix@C=4mm{
0\ar[r]&f_!\shc_{f^{-1}Z\times V}\ar[d]_{\varphi_{Z\times V}}\ar[r]&f_!\shc_{f^{-1}X\times V}\ar[d]_{\varphi_{X\times V}}\ar[r]&f_!\shc_{f^{-1}(\Omega\times V)}\ar[r]&0\\
0\ar[r]&\TH^S_X(\pOXS\otimes\C_{Z\times V})\ar[r]&\TH^S_X(\pOXS\otimes\C_{X\times V})\ar[r]&\TH^S_X(L)\ar[r]&0
}
\]
we define a morphism $\varphi_{\Omega\times V}:f_!\shc_{f^{-1}(\Omega\times V)}\to\TH^S_X(L)$ in $\mathsf{C}^+(\shd_{X\times S_{\R}/S})$. Functoriality with respect to $\shs$ follows from the description of the morphisms in~$\shs$ (\cf Section \ref{rc}).
\end{proof}

We can now end the proof of Proposition~\ref{Ldir}. According to Remark~\ref{C:functors}, it is enough to extend the morphism obtained in Lemma~\ref{LdirS} as a morphism of functors from $\rD^-(\shs)$ to $\rD^+(\shd_{X\times S_{\R}/S})$. Functoriality above leads to the definition of a morphism of functors from $\mathsf{C}^-(\shs)$ to the category of double complexes indexed by $\NN^2$ of $\Mod(\shd_{X\times S_{\R}/S})$, and we obtain the desired morphism of functors by passing to the associated simple complexes.
\end{proof}

\subsubsection{Behaviour of \texorpdfstring{$\RH^S$}{RHS} by localization and pullback with respect to \texorpdfstring{$X$}{X}}

\begin{proposition}\label{RHSloc}
Let $Y$ be a complex hypersurface of $X$. Then, for any $F\in\rD^\rb_\rc(\pOXS)$ there is a natural isomorphism
\begin{equation}\label{eq:RHSloc}
\RH^S_X(F)(*(\YS))\isom \RH^S_X(F\otimes \C_{(X\setminus Y)\times S}).
\end{equation}
In particular, if $F\in\rD^\rb_\cc(\pOXS)$,
\begin{enumerate}
\item\label{RHSloc1}
$\RH^S_X(F)(*(\YS))$ belongs to $\rD^\rb_\rhol(\DXS)$.

\item\label{RHSloc2}
There is a natural isomorphism $\RH^S_X(F\otimes \C_{\YS})\!\simeq\!R\Gamma_{[\YS]}(\RH^S_X(F))$ and so $R\Gamma_{[\YS]}(\RH^S_X(F))$ also belongs to $\rD^\rb_\rhol(\DXS)$.
\item\label{RHSloc3}
If the natural morphism $\RH^S_X(F)\to\RH^S_X(F)(*(\YS))$ is an isomorphism, then so is the natural morphism $F\otimes \C_{(X\setminus Y)\times S}\to F$.
\end{enumerate}
\end{proposition}

\begin{proof}
Let $f=0$ be a local defining equation of $Y$. Since this is a local problem we may start by assuming that $F=\pOXS\otimes\C_{\Omega\times S}$ for a relatively compact open subanalytic subset $\Omega$ of~$X$. Noting that $f$ is invertible on $ \tho(\C_{(\Omega\setminus Y)\times S}, \Db_{\XS})$, according to \cite[Prop.\,3.23]{Ka3}, the natural restriction morphism
\[
\tho(\C_{\Omega\times S}, \Db_{\XS})(*(\YS))\to \tho(\C_{(\Omega\setminus Y)\times S}, \Db_{\XS})
\]
is an isomorphism. Thus, applying Proposition \ref{P:AF}, the natural $\shd_{\overline{X}\times\overline{S}}$-linear morphism $\TH^S_X(F)(*(\YS))\to \TH^S_X(F\otimes \C_{(X\setminus Y)\times S})$ is an isomorphism for any $F\in\rD^\rb_\rc(\pOXS)$. The existence of the morphism \eqref{eq:RHSloc} and the fact that it is an isomorphism then follow by \eqref{RHS1} and functoriality.

The remaining statements~\eqref{RHSloc1} and \eqref{RHSloc2} follow straightforwardly (see also \cite[Ex.\,3.20]{MFCS2}), while \eqref{RHSloc3} is obtained by applying $\pSol_X$ to the isomorphism~\eqref{eq:RHSloc}, and by using Theorem \ref{TRHS}.\end{proof}

\begin{corollary}\label{C iminv}
For any $F\!\in\!\rD^\rb_\cc(\pOXS)$ and for any closed submanifold~$Y$ of $X$,
$R\Gamma_{[\YS]}(\RH^S_X(F))$ is a complex with regular holonomic $\DXS$-coho\-mologies.
\end{corollary}

\begin{proof}
The statement being local, we may assume that $Y$ is an intersection of smooth hypersurfaces of $X$ and then conclude by Proposition \ref{RHSloc}\eqref{RHSloc2} that
$R\Gamma_{[\YS]}(\RH^S(F))\simeq \RH^S(F\otimes \C_{\YS})$ which concludes the proof.
\end{proof}

\begin{proposition}\label{P2new}
Let $f:Y\to X$ be a smooth morphism
of complex manifolds. Then there exists a natural isomorphism in $\rD^\rb(f^{-1}\DXS)$, functorial in $F\in\rD^\rb_{\rc}(\pOXS)$:
\[
\Rhom_{\DYS}(\shd_{(Y\to X)\times S/S}, \RH^S_Y(f^{-1}F))\simeq f^{-1}\RH^S_X(F)[d_Y-d_X].
\]
\end{proposition}
The proof is performed by mimicking the proof of \cite[Th.\,5.8 (5.14)]{KS4} using Proposition \ref{Ldir}.

The following result is the relative version of \cite[Prop.\,5.9 (5.20)]{KS4} from which we adapt the proof.

\begin{proposition}\label{KSmorph}
For any $F\!\in\!\rD^\rb_\rc(\pOXS)$ and for any morphism \hbox{$f\!:\!Y\!\to\!X$} of complex manifolds, there exists a natural morphism in
$\rD^\rb(\DYS)$:
\begin{equation}\label{E:mor}
\Df^*\RH^S_X(F)[d_Y-d_X]\to \RH^S_Y(f^{-1}F).
\end{equation}
Moreover, when $F\in \rD^\rb_\cc(\pOXS)$, this morphism is an isomorphism.
\end{proposition}

\begin{proof}
We start by decomposing $f$ as the graph embedding $Y\to Y\times X$ followed by the projection $Y\times X\to X$, reducing to the case of (i) a closed immersion and (ii) a projection morphism.

Let us treat (i). We shall prove that
\[
\Df^*\RH^S_X(F)[d_Y-d_X]\simeq\RH^S_Y(f^{-1}F)
\]
by a natural isomorphism in $\rD^\rb(\DYS)$ functorially in $F$. We start by noticing that
\begin{equation}\label{eq:vanish}
\Df^*\RH^S_X(F\otimes\C_{(X\setminus Y)\times S})=0.
\end{equation}

To check this local statement we may assume, by induction on $\codim Y$, that~$Y$ is smooth of codimension one, and~\eqref{eq:vanish} follows from Proposition~\ref{RHSloc}. Hence we conclude that
\begin{align*}
\Df^*\RH^S_X(F)[d_Y&-d_X]\simeq\Df^*\RH^S_X(F\otimes \C_{\YS})[d_Y-d_X]\\
&\simeq \Df^*\Df_*\RH^S_Y(f^{-1}F)[d_Y-d_X]\quad\text{by Proposition \ref{RHSloc}\eqref{RHSloc2}}\\
&\simeq \RH^S_Y(f^{-1}F)\ \text{ since $\Df^*\Df_*[d_Y-d_X]\simeq\Id_{\rD^\rb(\DYS)}$}.
\end{align*}

Let us now treat (ii). Let $f:Y\to X$ a projection of $Y=X\times Z$ on~$X$. Recall that in that case we have a natural transformation of functors on $\rD^\rb(\DYS)$
\[
\shd_{(Y\to X)\times S/S} \otimes_{f^{-1}\DXS}\Rhom_{\DYS}(\shd_{(Y\to X)\times S/S},\cdot)\to\Id.
\]
Then by Proposition~\ref{P2new} we obtain the canonical morphism \eqref{E:mor}.

Let us now assume that $F\in \rD^\rb_\cc(\pOXS)$. Since the result is true when~$f$ is a closed embedding, it remains to consider the case where $f$ is a smooth morphism. Then $f$ is a non-characteristic morphism for any $\shm\in\rD^\rb_{\coh}(\shd_X)$ in the sense of \cite[Def.\,11.2.11]{KS1}, hence, according to Theorem 11.3.5 of \loccit, we have a functorial isomorphism $f^{-1}\Sol(\shm)\simeq \Sol(\Df^*\shm)$. If~$\shm\in\rD^\rb_{\rhol}(\shd_X)$ and $F=\Sol(\shm)$, according to the Riemann-Hilbert correspondence in the absolute case (here denoted by $\RH$), we have $\shm\simeq \RH(F)$ and we conclude an isomorphism $\Df^*\shm\simeq \RH(f^{-1}F)$. As a consequence,
\begin{align*}
Li_s^*\Df^*\RH^S_X(F)&\simeq \Df^*Li_s^*\RH^S_X(F) \simeq
\Df^*\tho (Li_s^*F, \sho_X)[d_X] \\
&\overset{(*)}\simeq \tho(f^{-1}Li_s^*F,\sho_Y)[d_X]\simeq Li_s^*\RH^S_Y(f^{-1}F)[d_X-d_Y],
\end{align*}
where $(*)$ holds by the absolute case recalled above and the compatibility of our constructions with the similar ones in the absolute case. By applying the variant of Nakayama's Lemma (see \ref{subsec:prelimD}) to the morphism~\eqref{E:mor} in $\rD^\rb_{\rhol}(\DYS)$, we obtain that \eqref{E:mor} is an isomorphism for any smooth morphism, and thus for any morphism.
\end{proof}

\subsection{Behaviour of \texorpdfstring{$\RH^S$}{RHS} under finite ramification over \texorpdfstring{$S$}{S}}\label{s5}

Let $s_0\in S$, let $N$ be a natural number and let $\delta: (S',s'_0)\to (S,s_0)$ be the ramification of center $s_0$ of
degree $N$ (that is, there exist a local chart on $S$ centered in $s_0$ and a local chart centered in $s'_0$ such that $\delta(s')=s'^N$). For simplicity we shall keep the notation $\delta$ also to denote the morphism $\Id_X\times \delta: X\times S'\to X\times S$.

For a $\DXS$-module $\shm$, \resp an object $F\in \rD^\rb_{\rc}(\pOXS)$, the pullback is defined by
\[
\Ddelta^*\shm:=\sho_{\XS'}\otimes_{\delta^{-1}\sho_{\XS}}\delta^{-1}\shm\quad\resp \delta^*F:=p^{-1}\sho_{S'}\otimes_{p^{-1}\delta^{-1}\sho_S} \delta^{-1}F.
\]

We remark that if $\Ddelta_*$ denotes the direct image in the sense of $\DXS$-modules, then $\Ddelta_*=\delta_*$, so we simply denote $\Ddelta^*,\Ddelta_*$ by $\delta^*,\delta_*$. We also remark that $\sho_{\XS'}$ is flat over $\delta^{-1}\sho_{\XS}$ hence we have
\[
\delta_*\delta^*\shm\simeq\delta_*\sho_{\XS'}\otimes_{\sho_{\XS}}\shm\quad\resp \delta_*\delta^*F:=\delta_*\sho_{S'}\otimes_{\sho_S}F,
\]
so that $\shm$, \resp $F$, is a direct summand in $\delta_*\delta^*\shm$, \resp in $\delta_*\delta^*F$.

The first pullback induces a well-defined exact functor from $\rD^\rb_{\coh}(\DXS)$ to $\rD^\rb_{\coh}(\shd_{X\times S'/S'})$, as already used in \cite{MFCS2} in a particular situation (proof of Corollary 2.8, where $\delta$ is denoted by $\rho$), and the second one a well-defined functor
$\rD^\rb_{\rc}(\pOXS)\to \rD^\rb_{\rc}(p^{-1}\sho_{S'})$. The following results benefited from useful discussions with Luca Prelli.

\refstepcounter{theorem}
\begin{lemma}\label{deltat}
There is a well-defined morphism in $\rD^\rb(\rho'_{S *}\delta^{-1}\DXS)$
\begin{equation}\label{E6}
\delta^{-1}\sho_{\XS}^{t,S}\to \sho_{\XS'}^{t,S'}.
\end{equation}
\end{lemma}
\begin{proof}
It is sufficient to prove the existence of such a morphism when replacing $\sho^{t,S}_{\XS}$ with $C^{\infty,t,S}_{\XS}$. In that case it is obtained by the composition of functions with $\delta$ --- this does not interfere with the growth conditions --- yielding a $\delta^{-1}\DXS$-linear morphism since the operators in $\DXS$ do not involve derivations along $S$.
\end{proof}

\begin{proposition}\label{ram}
With the notation as above, for any $F\in\rD^\rb_{\rc}(\pOXS)$, there exists a morphism, functorial in $F$,
in $\rD^\rb(\shd_{X\times S'/S'})$
\[
\Psi_F: \delta^*\RH^S_X(F)\to \RH^{S'}_X(\delta^*F),
\]
which is an isomorphism if $F$ is an object of $\rD^\rb_{\cc}(\pOXS)$.
\end{proposition}

\begin{proof}

Following the definition of the relative Riemann-Hilbert functor, we have to construct a natural
morphism:
\begin{multline}\label{E3}
\sho_{\XS'}\otimes _{\delta^{-1}\sho_{\XS}}\delta^{-1}{{\rho'}_{S}^{ -1}}\Rhom_{\rho'_{S *}\pOXS}(\rho'_{{S} *}F,\sho^{t,S}_{\XS})\\
\to\rho^{\prime-1}_{S'}\Rhom_{\rho'_{S' *}p^{-1}\sho_{S'}}(\rho'_{ S' *}\delta^*F, \sho^{t,S'}_{\XS'}).
\end{multline}

We have a natural isomorphism of functors on sites $\delta^{-1}\rho^{\prime-1}_S\simeq \rho^{\prime-1}_{S'}\delta^{-1}$ which yields a natural isomorphism in $\rD^\rb(\delta^{-1}\DXS)$
\begin{multline}\label{E4}
\delta^{-1}{\rho'^{-1}_{S}}\Rhom_{\rho'_{S *}\pOXS}(\rho'_{S *}F,\sho^{t,S}_{\XS})\\
\simeq {\rho^{\prime-1}_{S'}}\delta^{-1}\Rhom_{\rho'_{S *}\pOXS}(\rho'_{S *}F,\sho^{t,S}_{\XS}).
\end{multline}

Recall that, for a morphism $g: Z'\to Z$ of manifolds and $\sha$ a sheaf of rings on $Z$, we have a natural morphism of bifunctors on $\rD^\rb(\sha)$ (\cf \cite[(2.6.27]{KS1}):
\begin{equation}\label{eq:fA}
f^{-1}\Rhom_{\sha}(\cbbullet, \cbbullet)\to \Rhom_{f^{-1}\sha}(f^{-1}(\cbbullet), f^{-1}(\cbbullet)).
\end{equation}

Since we are working with sheaves on Grothendieck topologies (see \cite{EP} and~\cite{KS5}), we have the analogous of \eqref{eq:fA}, that is, in the present situation, we have a natural morphism
in $\rD^\rb(\rho'_{S' !}\delta^{-1}\DXS)$
\begin{equation*}
\delta^{-1}\!\!\Rhom_{\rho'_{ S *}\pOXS}(\rho'_{S *}F,\sho^{t,S}_{\XS})
\to \Rhom_{\delta^{-1}\rho'_{S *}\pOXS}(\delta^{-1}\rho'_{S *}F,\delta^{-1}\sho^{t,S}_{\XS}),
\end{equation*}
hence a natural morphism
in $\rD^\rb(\delta^{-1}\DXS)$
\begin{multline}\label{E5}
{\rho^{\prime-1}_{S'}}\delta^{-1}\Rhom_{\rho'_{S *}\pOXS}(\rho'_{S *}F,\sho^{t,S}_{\XS})\\
\to {\rho^{\prime-1}_{S'}}\Rhom_{\delta^{-1}\rho'_{S *}\pOXS}(\delta^{-1}\rho'_{S *}F,\delta^{-1}\sho^{t,S}_{\XS}).
\end{multline}

According to the morphism \eqref{E6} in
$\rD^\rb(\delta^{-1}\DXS)$, combining with \eqref{E5} and the commutation $\delta^{-1}\rho'_{S *}\simeq \rho'_{S' *}\delta^{-1}$
we derive a functorial chain of morphisms in $\rD^\rb(\delta^{-1}\DXS)$
\begin{multline*}
\rho^{\prime-1}_{S'}\delta^{-1}\Rhom_{\rho'_{S *}\pOXS}(\rho'_{S *}F,\sho^{t,S}_{\XS})\\
\to \rho^{\prime-1}_{S'}\Rhom_{\delta^{-1}\rho'_{ S *}\pOXS}(\delta^{-1}\rho'_{S *}F,\sho^{t,S'}_{\XS'})\\
\hspace*{2.3cm}\to \rho^{\prime-1}_{S'}\Rhom_{\rho'_{S *}\delta^{-1}\pOXS}(\rho'_{S *}\delta^{-1}F,\sho^{t,S'}_{\XS'})\\
\to \rho^{\prime-1}_{S'}\Rhom_{\rho'_{S' *}p^{-1}\delta^*\sho_S}(\rho'_{S' *}\delta^*F,\sho^{t,S'}_{\XS'}),
\end{multline*}
where the last term results by applying $\sho_{S'}\otimes_{\delta^{-1}\sho_S}(\cdot)$. We also
remark that the last term (which we will name $\shl$ for simplicity) is the right term of the desired morphism \eqref{E3} and is already an object of $\rD^{\rb}(\shd_{\XS'})$.
Hence, by applying to \eqref{E4} the functor $\sho_{\XS'}\otimes_{\delta^{-1}\sho_{\XS}}(\cdot)$ we derive a chain of natural morphisms
in $\rD^\rb(\shd_{X\times S'/S'})$.
\begin{multline*}
\sho_{\XS'}\otimes _{\delta^{-1}\sho_{\XS}}\delta^{-1}\rho^{\prime-1}_{S'}\Rhom_{\rho'_{S *}\pOXS}(\rho'_{S *}F,\sho^{t,S}_{\XS})\\
\to \sho_{\XS'}\otimes _{\delta^{-1}\sho_{\XS}}\shl\to\shl
\end{multline*}
whose composition gives the desired morphism $\Psi_F$.

Assume now that $F\in \rD^\rb_{\cc}(\pOXS)$ and let us prove that $\Psi_F$ is an isomorphism.

Since in that case $\RH^S_X(F)$ is an object of $\rD^\rb_{\rhol}(\DXS)$, it is clear that $\delta^*\RH^S_X(F)$ is an object of $\rD^\rb_{\rhol}(\DXS')$.
The same holds true with $\RH^{S'}_X(\delta^*F)$ since $\delta^*F$ is an object of $\rD^\rb_{\cc}(p^{-1}_X\sho_{S'})$.

It is then sufficient to apply $Li^*_{s'}$ to both sides of the morphism $\Psi_F$ and apply Proposition \ref{com}, noting that $Li^*_{s'}\delta^*=Li^*_s$, where $s\!=\!\delta(s')$. This way, both members become, by reduction to the absolute case, isomorphic to $\tho(Li^*_sF, \sho_X)[d_X]$.
\end{proof}

We shall now come back to the situation described at the beginning of this section, and we keep the same notations.

\begin{proposition}\label{ram4}
Let $\shm, \shn$ be $\DXS$-modules and assume that $\shm$ is coherent. If the complex $\Rhom_{\shd_{X\times S'/S'}}(\delta^*\shm, \delta^*\shn)$ belongs to $\rD^\rb_{\cc}(p^{-1}_X\sho_{S'})$, then the complex $\Rhom_{\DXS}(\shm, \shn)$ also belongs to $\rD^\rb_{\cc}(\pOXS)$.\end{proposition}

\begin{proof}
Since $\shn$ is a direct summand of $\delta_*\delta^*\shn$ it suffices to deduce from the assumption that $\Rhom_{\DXS}(\shm,\delta_*\delta^{*}\shn)$ is an object of $\rD^\rb_{\cc}(p_X^{-1}\sho_S)$. Thanks to the adjunction morphism we have in $\rD^\rb(\pOXS)$
\begin{align*}
\Rhom_{\DXS}(\shm, \delta_*\delta^*\shn)
&\simeq\delta_*\Rhom_{\delta^{-1}\DXS}(\delta^{-1}\shm, \delta^*\shn)\\
&\simeq\delta_*\Rhom_{\shd_{X\times S'/S'}}(\delta^*\shm, \delta^*\shn).
\end{align*}
The result is then a consequence of the assumption and the properness of~$\delta$.
\end{proof}

The following result will be used in Section~\ref{S3}. Recall (\cf\cite[(3.17)]{MFCS2}) that, for $F\in \rD^\rb_\rc(\pOXS)$, there exists a natural morphism
\[
\RH^S_X(F)\to\Rhom_{\pOXS}(F,\sho_{\XS})[d_X].
\]

\begin{lemma}\label{D-type 5}
Let $\shl$ be a $\DXS$-module of D-type and
let $F$ be an object of $\rD^\rb_\rc(\pOXS)$ be such that $F\simeq F\otimes \C_{(X\setminus D)\times S}$.
Then the morphism
\begin{multline*}
\beta_{\shl, F}: \Rhom_{\DXS}(\shl, \RH^S_X(F))\\
\to\Rhom_{\DXS}(\shl, \Rhom_{\pOXS}(F,\sho_{\XS})[d_X])
\end{multline*}
is an isomorphism.
\end{lemma}

\begin{proof}
The case where $\shl=\sho_{\XS}$ was proved in \cite[Lem.\,3.19]{MFCS2}. We aim at reducing to this case. The statement has a local nature so we choose local coordinates $x_1,\dots,x_n$ in~$X$ such that $D=\{x\in X\mid x_1\cdots x_\ell=0\}$, and we assume that $S$ is a disc of small enough radius with coordinate $s$. Let $\delta:S'\to S$ be a finite morphism ramified at $s=0$ only. According to Propositions~\ref{ram} and~\ref{ram4}, the assertion of the lemma holds for $(\shl,F)$ if it holds for $(\delta^*\shl,\delta^*F)$. By Theorem \ref{th:Dtype}\eqref{th:Dtype2} and the same argument as in the proof of \cite[Cor.\,2.8]{MFCS2}, there exists $\delta$ such that $\delta^*\shl$ is isomorphic to
\[
\shd_{X\times S'/S'}\Big/\Bigl(\textstyle\sum_{i=1}^\ell\shd_{X\times S'/S'}(x_i\partial_{x_i}-\alpha_i)+\sum_{i=\ell+1}^n\shd_{X\times S'/S'} \partial _{x_i}\Bigr)
\]
for some holomorphic functions $\alpha_i$, $i=1,\dots,\ell$, on $S'$. We will therefore assume that $\shl$ is already of this form. We can replace $F$ with a resolution as given by Proposition \ref{P:AF} and, since the result is local, we may reduce to the case $F=\pOXS\otimes\C_{(\Omega\setminus D)\times S}$ (since $F\simeq F\otimes \C_{(X\setminus D)\times S}$) for some relatively compact open subanalytic subset $\Omega$ of $X$. We have
\[
\RH^S(F)\simeq \tho(\C_{(\Omega\setminus D)\times S}, \sho_{\XS})[d_X]
\]
and
\[
\Rhom _{\pOXS}(F, \sho_{\XS})[d_X]\simeq
\Rhom (\C_{(\Omega\setminus D)\times S}, \sho_{\XS})[d_X].
\]
Let us consider the automorphism $\Phi$ induced on $\tho(\C_{(\Omega\setminus D)\times S}, \Db_{\XS})$ and on $\Rhom (\C_{(\Omega\setminus D)\times S}, \Db_{\XS})\simeq \Gamma_{(\Omega\setminus D)\times S}(\Db_{\XS})$
by multiplication by the real analytic function $|x_1|^{2 \alpha_1(s)}|x_2|^{2\alpha_2(s)}\cdots |x_\ell|^{2\alpha_\ell(s)}$.
Then, in $\rD^\rb(\pOXS)$, $\Phi$ induces isomorphisms\vspace*{-3pt}
\begin{multline*}
\Phi_1:\Rhom_{\DXS}(\shl, \tho(\C_{(\Omega\setminus D)\times S}, \Db_{\XS}))\\
\isom \Rhom_{\DXS}(\sho_{\XS}, \tho(\C_{(\Omega\setminus D)\times S}, \Db_{\XS}))
\end{multline*}
and\vspace*{-3pt}
\begin{multline*}
\Phi_2:\Rhom_{\DXS}(\shl, \Gamma_{(\Omega\setminus D)\times S}(\Db_{\XS}))\\
\isom\Rhom_{\DXS}(\sho_{\XS}, \Gamma_{(\Omega\setminus D)\times S} (\Db_{\XS})).
\end{multline*}
We derive a morphism in $\rD^\rb(\pOXS)$\vspace*{-3pt}
\begin{multline*}
\Phi_2\mu\Phi_1^{-1}:\Rhom_{\DXS}(\sho_{\XS}, \tho(\C_{(\Omega\setminus D)\times S}, \Db_{\XS}))\\\to\Rhom_{\DXS}(\sho_{\XS},\Gamma_{(\Omega\setminus D)\times S} (\Db_{\XS}))
\end{multline*}
which coincides with the natural one. By the realification procedure (\cf Sec\-tion~\ref{subsubsec:real}), we are thus reduced to the case $\shl=\sho_{\XS}$, as wanted.
\end{proof}

\section{Relative Riemann-Hilbert correspondence}\label{S3}

Proving Theorem \ref{RHH} is equivalent to proving Theorem \ref{Tequivtf}. Indeed, in one direction, let us recall the method introduced in \cite[\S 4.3]{MFCS2} to deduce Theorem~\ref{RHH} from Theorem~\ref{Tequivtf}. According to \cite[(3.17)]{MFCS2}, there exists a natural morphism of bifunctors from $\rD^\rb_{\rhol}(\DXS)^\mathrm{op}\times \rD^\rb_{\rc}(\pOXS)$ to
$ \rD^\rb(\pOXS)$:\vspace*{-3pt}
\begin{multline}\label{E:20}
\Rhom_{\DXS}(\shm, \RH^S_X(F))\\
\to \Rhom_{\DXS}(\shm, \Rhom_{\pOXS}(F, \sho_{X\times S}[d_X] ))\\
\hfil \simeq \Rhom_{\pOXS}(F, \pSol_X(\shm)),
\end{multline}
where the last isomorphism is an application of \cite[(2.6.7)]{KS1}. Notice that the right-hand side of \eqref{E:20} is an object of $\rD^\rb_{\cc}(\pOXS)$ provided that $F\in \rD^\rb_{\cc}(\pOXS)$. In that case, by Theorem~\ref{Tequivtf}, the left-hand side is also an object of $\rD^\rb_{\cc}(\pOXS)$. In particular, $\Rhom_{\DXS}(\shm, \RH^S_X(F))_{(x,s)}$ has $\sho_{S,s}$-finitely generated cohomologies for any $(x,s)\in X\times S$. By the variant of Nakayama's lemma recalled in  Section \ref{subsubsec:relativeCc}, and since $Li_{s_o}^*\eqref{E:20}$ is an isomorphism for any $s_o\in S$ (this is the absolute case, where the result is known (cf.\,\cite[Cor.\,8.6]{Ka3}), we conclude that \eqref{E:20} is an isomorphism. Replacing~$F$ with $\pSol_X\shm$, we deduce an isomorphism of functors\vspace*{-3pt}
\[
\Id_{\rD^\rb_{\rhol}(\DXS)}\To{\beta} \RH^S_X \circ \pSol_X,
\]
concluding Theorem \ref{RHH} as explained in the introduction.

Conversely, Theorem \ref{RHH} implies Theorem \ref{Tequivtf} since the former implies full faithfulness of $\pSol$, so we have a natural isomorphism\vspace*{-3pt}
\[
\Rhom_{\DXS}(\shm, \RH^S_X(F))\simeq\Rhom_{\pOXS}(F, \pSol_X\shm)
\]
and the right-hand side belongs to $\rD^\rb_{\cc}(\pOXS)$.

\subsection{Proof of Theorems \ref{RHH} and \ref{Tequivtf} in the torsion case}\label{sec:TRHE}

Recall that, according to Proposition~\ref{Charhol} a holonomic $\DXS$-module $\shm$ is torsion if and only if $\supp(M)\subseteq X\times T$ with $\dim T=0$. In that case we have the following result.

\begin{proposition}\label{L:tor0}
Let $\shm\in \mathrm{Mod}_{\rhol}(\DXS)$ be a torsion $\DXS$-module. Then $\tilde{\shm}:=\shd_{\XS}\otimes_{\DXS}\shm$ is a regular holonomic $\shd_{\XS}$-module.
\end{proposition}

\begin{proof}
The statement being local, we may assume that $\Char(\shm)=\Lambda\times \{s_o\}$, where $\Lambda$ is a Lagrangian $\C^*$-conic closed analytic subset in $T^*X$, and, taking a local coordinates $s$ on $S$ vanishing at $s_o$, there exists $n\in\N$ such that $s^n\shm=0$. Since we are dealing with triangulated categories, by an easy argument by induction on $n$ we may assume that $n=1$. In that case, we have $\shm\simeq \shm_0\boxtimes \sho_S/\sho_Ss$, where, by the assumption of relative regularity, $\shm_0$ is a regular holonomic $\shd_X$-module satisfying $\Char(\shm_0)=\Lambda$. By construction $\tilde{\shm}\simeq \shm_0\boxtimes \shd_S/\shd_Ss$ and $\Char(\tilde{\shm})=\Lambda\times T^*_{T}S=:\tilde{\Lambda}$.

Therefore $\tilde{\shm}$ is a regular holonomic $\shd_{\XS}$-module since the category of regular holonomic $\shd_{\XS}$-modules is closed under external tensor product.
\end{proof}

We denote by
$\rD^\rb_{\rhol}(\DXS)_\rt$ the thick subcategory of $\rD^\rb_{\rhol}(\DXS)$
whose objects have support in $X\times T$ with $\dim T=0$.

\begin{proposition}\label{Tequiv1}
The solution functor $\pSol$ restricted to $\rD^\rb_{\rhol}(\DXS)_\rt$ is an equivalence of categories
\[
\pSol_X:\rD^\rb_{\rhol}(\DXS)_\rt\to \rD^\rb_{\cc}(\pOXS)_\rt
\]
with quasi-inverse the restriction of the functor
$\RH^S_X$ to $\rD^\rb_{\cc}(\pOXS)_\rt$.
\end{proposition}

\begin{proof}
It is sufficient to prove that the restriction of
$\RH^S_X$ to $\rD^\rb_{\cc}(\pOXS)_{t }$ is fully faithful.
Indeed $\pSol$ is essentially surjective since, according to Theorem \ref{TRHS}, for any
$F \in\rD^\rb_{\cc}(\pOXS)$ we have $F\simeq \pSol_X \RH^S_X(F)$,
and in the case of a torsion object $F$ in $\rD^\rb_{\cc}(\pOXS)_\rt$ we have
$\RH^S_X(F)\in \rD^\rb_{\rhol}(\DXS)_\rt$.

For the full faithfulness it is enough to prove that the morphism:
\begin{multline*}
\Rhom_{\DXS}(\shm, \RH^S_X(G))\\
\to\Rhom_{\DXS}(\shm, \Rhom_{\pOXS}(G,\sho_{\XS})[d_X])
\end{multline*}
is an isomorphism for any $\shm\!\in\!\rD^\rb_{\rhol}(\DXS)_\rt$ and for any $G\!\in\!\rD^\rb_\rc(\pOXS)$.

The cohomologies of $\shm$ are regular holonomic $\DXS$-modules and,
accor\-ding to Proposition~\ref{L:tor0},
$\shd_{\XS}\otimes_{\DXS}\shm$ is a complex whose cohomologies are regular holonomic.

Thanks to Proposition \ref{P:AF}, we may assume that
\hbox{$G\!=\!\pOXS\otimes\C_{\Omega\times S}$}
for some open subanalytic subset $\Omega$ of $X$, hence
\[
\RH^S_X(G)=\tho(\C_{\Omega\times S}, \sho_{\XS})[d_X],
\]
which is a complex with $\shd_{\XS}$-modules as cohomologies and we get a chain of isomorphisms
\begin{align*}
\Rhom_{\DXS}&(\shm, \RH^S_X(G))\simeq \Rhom_{\shd_{\XS}}(\shd_{\XS}\otimes_{\DXS}\shm, \RH^S_X(G))\\
&\overset{(*)}{\simeq}
\Rhom_{\shd_{\XS}}(\shd_{\XS}\otimes_{\DXS}\shm, \Rhom(\C_{\Omega\times S},\sho_{\XS})[d_X])\\
&\simeq \Rhom_{\DXS}(\shm, \Rhom(\C_{\Omega\times S},\sho_{\XS})[d_X])\\
&\simeq \Rhom_{\DXS}(\shm, \Rhom_{\pOXS}(G,\sho_{\XS})[d_X]),
\end{align*}
where isomorphism $(*)$ follows by \cite[Cor.\,8.6]{Ka3}.
\end{proof}

\subsection{Proof of Theorem \ref{Tequivtf}, main propostion}\label{subsec:proofmain}

For a regular holonomic $\DXS$-module $\shm$, let us set (see Proposition~\ref{Charhol})
\begin{equation}\label{eq:strictperv2}
\begin{gathered}
\Char(\shm):=\bigcup_{j}\Char(\shh^j\shm)=\bigcup_{i\in I}\Lambda_i\times\nobreak T_i\\
\supp(\shm)=\bigcup_{i\in I}Y_i\times T_i,\quad \supp_X(\shm)=Z=Z_{\shm}:= \bigcup_{i\in I}Y_i.
\end{gathered}
\end{equation}

\begin{proposition}\label{P:strictperv2}
Let $\shm$ be a strict regular holonomic $\DXS$-module with $X$-support $Z$. Let $Y\subset X$ be a hypersurface containing the singular locus $\mathrm{Sing}(Z)$ and all subsets $Y_i$ with $\dim Y_i<\dim Z$. Then the localized $\DXS$-module $\shm(*(\YS))$ is regular holonomic and locally isomorphic to the projective pushforward of a relative $\shd$-module of D\nobreakdash-type.
\end{proposition}

\begin{proof}
The question is local. The assumption on $Y$ implies that $Z^o:=Z\moins(Y\cap Z)$ is smooth of pure dimension $\dim Z$ and the characteristic variety of $\shm_{|(X\moins Y)\times S}$ is contained in $(T^*_{Z^o}X)\times S$. By Kashiwara's equivalence, $\shm_{|(X\moins Y)\times S}$ is the pushforward by the inclusion map of a coherent $\sho_{Z^o\times S}$-module with flat relative connection. The strictness assumption entails that this flat relative connection is of the form $(\sho_{Z^o\times S}\otimes_{\pOS} F,\rd_{Z^o\times S/S})$ for some locally constant $p^{-1}_{Z^o}\sho_S$-module $F$ which is  \emph{locally free of finite rank}.

One can find a complex manifold~$X'$ together with a divisor with normal crossings $Y'\subset X'$ and a projective morphism $\pi:X'\to X$ which induces a biholomorphism \hbox{$X'\moins Y'\isom Z^o$}. We~set $\delta=\dim Z-\dim X=\dim X'-\dim X\leq0$. For each $\ell$, we consider the $\DXpS$-module $\shm^{\prime\ell}:=\shh^\ell\Dpi^*\shm$. Although we cannot yet conclude it is coherent, the latter is locally an inductive limit (union) of coherent $\DXpS$-submodules, and also of $\sho_{\XpS}$-coherent submodules (\cf \cite[Prop.\,2.1]{L-S87} and its proof). We simply say that $\shm^{\prime\ell}$ is quasi-coherent (over $\DXpS$ or over $\sho_{\XpS}$). We will use the following property, that is deduced from the similar one for coherent $\sho_{\XpS}$-modules:
\begin{itemize}
\item[$(*)$]
A quasi-coherent $\sho_{\XpS}$-module which is zero on $(X'\moins Y')\times S$ becomes zero after being tensored with $\sho_{\XpS}(*(\YpS))$.
\end{itemize}

If $\ell\neq\delta$, the sheaf-theoretic restriction of $\shm^{\prime\ell}$ to $(X'\moins Y')\times S$ is zero, so $\shm^{\prime\ell}(*(\YpS))=0$ owing to quasi-coherence, according to $(*)$. Since $\sho_{\XpS}(*(\YpS))$ is flat over $\sho_{\XpS}$, we conclude that
\begin{equation}\label{eq:DpiMloc}
\Dpi^*\bigl(\shm(*(\YS))\bigr)[\delta]\simeq\bigl(\Dpi^*\shm\bigr)(*(\YpS))[\delta]\simeq\shm^{\prime\delta}(*(\YpS)).
\end{equation}
We will first check that $\shm^{\prime\delta}(*(\YpS))$ is strict (\ie \hbox{$Li_s^*\shm^{\prime\delta}(*(\YpS))$} has cohomology in degree zero only, \cf\cite[Lem.\,1.13]{MFCS2}). Strictness of $\shm(*(\YS))$ follows from flatness of $\sho_{\XS}(*(\YS))$ over $\sho_{\XS}$. Furthermore, as a complex of $\sho_{\XpS}$-modules, $\Dpi^*\bigl(\shm(*(\YS))\bigr)$ is nothing but \hbox{$L\pi^*\bigl(\shm(*(\YS))\bigr)$}. We then have, for each $s\in S$,
\begin{align*}
Li^*_s\shm^{\prime\delta}(*(\YpS))&\simeq Li^*_sL\pi^*\bigl(\shm(*(\YS))\bigr)[\delta]\quad (\text{according to \eqref{eq:DpiMloc}})\\
&\simeq L\pi^*(i^*_s\shm)(*Y)[\delta]\quad (\text{strictness of }\shm(*(\YS)))\\
&\simeq L^\delta\pi^*(i^*_s\shm)(*Y)\quad (\text{same argument as \eqref{eq:DpiMloc}})
\end{align*}
has cohomology in degree zero only, as wanted. The same argument shows that, while $\shm^{\prime\delta}(*(\YpS))$ may a~priori be non $\DXpS$-coherent, its restriction by $i^*_s$ is regular holonomic (hence $\shd_{X'}$\nobreakdash-coherent) for each $s\in S$.

We now take up the argument of \cite[Proof of Prop.\,2.11]{MFCS2} and show that $\shm^{\prime\delta}(*(\YpS))$ is regular holonomic and of D-type with respect to $Y'$. As noticed at the beginning of the proof, $F:=\ho_{\DXpS}(\sho_{\XpS}, \shm^{\prime\delta})|_{(X'\setminus Y')\times S}$ is locally free of finite rank. Let \hbox{$j':X'\moins Y'\hto X'$} denote the inclusion. The isomorphism
\[
j^{\prime-1}\shm^{\prime\delta}\isom(\sho_{(X'\setminus Y')\times S}\otimes_{\pOS}F,\rd_{\XpS/S})=:(V,\nabla)
\]
extends as a morphism of $\DXpS(*(\YpS))$-modules
\[
\psi:\shm^{\prime\delta}(*(\YpS))\to j'_*(V,\nabla).
\]
Let $m$ be a local section of $\shm^{\prime\delta}(*(\YpS))$. Since for each $s\in S$, $i^*_s\bigl(\shm^{\prime\delta}(*(\YpS))\bigr)=(i^*_s\shm^{\prime\delta})(*Y')$ is regular holonomic, the image $m(\cdot,s)$ of $m$ in the latter module has moderate growth in the sense of \cite[p.\,862]{K-K81} when restricted to $X'\moins Y'$. According to \cite[Lem.\,2.12]{MFCS2}, $\psi(m)$~is a local section of the Deligne extension $\wt V$ of $(V,\nabla)$, which is $\DXpS$-coherent by Theorem~\ref{th:Dtype}\eqref{th:Dtype1}. Then $\im\psi$, being quasi-coherent, is a coherent $\DXpS$-submodule of~$\wt V$. By applying~$(*)$ to the kernel and cokernel of $\psi$, we obtain that $\psi$ is an isomorphism.

According to Proposition \ref{Dirim}, $\Dpi_*\wt V$ has regular holonomic cohomology. Furthermore, since $\shh^j\Dpi_*\wt V$ is supported on $\YS$ for $j\neq0$, and since $\wt V=\wt V(*(\YpS))$, so that $\Dpi_*\wt V\simeq\Dpi_*\wt V(*(\YS))$, we have
\[
\Dpi_*\wt V\simeq \shh^0\Dpi_*\wt V\simeq\shh^0\Dpi_*\wt V(*(\YS)).
\]
On the other hand, there is a natural adjunction morphism (\cf \cite[Lem.\,4.28 \& Prop.\,4.34]{Ka2})
\[
\Dpi_*\Dpi^*\shm[\delta]\to\shm,
\]
which induces a morphism of coherent $\DXS(*(\YS))$-modules
\[
\shh^0\Dpi_*\wt V\simeq(\shh^0\Dpi_*\shm^{\prime\delta})(*(\YS))\to\shm(*(\YS)),
\]
where the left-hand side is $\DXS$-coherent and regular holonomic. Its cokernel is zero on $(X\moins Y)\times S$ and $\DXS(*(\YS))$-coherent, hence it is zero according to $(*)$, so that this morphism is an isomorphism. In~conclusion, $\shm(*(\YS))$ is regular holonomic.
\end{proof}

After the proof of Theorem \ref{RHH}, Proposition \ref{P:strictperv2} can be improved:

\begin{corollary}[of Theorem \ref{RHH}]\label{C:ststY}
For any $\shm$ in $\rD^{\rb}_{\rhol}(\DXS)$ and any hyper\-surface $Y\subset X$, the complexes
$R\Gamma_{[\YS]}(\shm)$ and $ \shm(*(\YS))$ belong to $\rD^{\rb}_{\rhol}(\DXS)$.
\end{corollary}

\begin{proof}
In view of the equivalence of Theorem \ref{RHH}, this reduces to Proposition~\ref{RHSloc}.
\end{proof}

\subsection{End of the proof of Theorem \ref{inverseimage}}
We can argue by induction on the length of~$\shm$ and then reduce to the cases of a projection and of a closed embedding. The first case was proved in Section \ref{subsubsec:smoothpullbackreghol}. The case of a closed embedding $i:Y\hto X$ is a consequence of Corollary \ref {C:ststY}.\qed

\subsection{End of the proof of Theorem \ref{Tequivtf}}\label{subsec:endTequivtf}
We refer to \cite[Lem.\,4.1.4]{KS6} which contains the guidelines for the proof of Theorem~\ref{Tequivtf}. In what follows, for a complex manifold~$X$ and $\shm\in\rD^\rb_{\rhol}(\DXS)$ we consider the statement
\[
P_X(\shm):\quad \Rhom_{\DXS}(\shm, \RH^S_X(F))\in\rD^\rb_{\cc}(\pOXS)\ \ \forall F\in\rD^\rb_{\cc}(\pOXS),
\]
in other words, $\shm$ satisfies Theorem \ref{Tequivtf}.

\refstepcounter{theorem}
\begin{lemma}\label{Luisa}
The statement $P$ satisfies the following properties.
\begin{enumerate}
\item \label{a}
For any manifold $X$ and any open covering $(U_i)_{i\in I}$ of $X$,
\[
P_X(\shm)\iff P_{U_i}(\shm_{|U_i})\ \forall i\in I.
\]
\item\label{b} $P_X(\shm)\implies P_X(\shm[n])\ \forall n\in\ZZ$.
\item\label{c}
For any distinguished triangle $\shm'\to \shm\to \shm'' \To{+1}$ in $\rD^{\rb}_{\hol}(\DXS)$,
\[
P_X(\shm')\wedge P_X(\shm'')\implies P_X(\shm).
\]
\item\label{d}
For any regular relative holonomic $\DXS$-modules $\shm$ and $\shm'$,\[
P_X(\shm\oplus\shm')\implies P_X(\shm).
\]
\item \label{e}
For any projective morphism $f:X\to Y$ and any regular holonomic $\DXS$-module $\shm$ which is $f$-good,
\[
P_X(\shm)\implies P_Y(\Df_*\shm).
\]
\item\label{g}
If $\shm=\shh^0(\shm)$ is torsion, then $P_X(\shm)$ is true.
\end{enumerate}
\end{lemma}

\begin{proof}
It is clear that $P_X(\cbbullet)$ satisfies Properties \ref{Luisa}\eqref{a}, \eqref{b}, \eqref{c}, \eqref{d}. Then Property \eqref{e} follows by adjunction, Proposition~\ref{KSmorph} and by the stability of $S$-$\C$-constructibility under proper direct image. Last, Property \eqref{g} has been seen in Section \ref{sec:TRHE}.
\end{proof}

\begin{proof}[End of the proof of Theorem \ref{Tequivtf} (and hence that of Theorem \ref{RHH})]
We wish to prove that $P_X(\shm)$ is true for any $X$ and $\shm\in \rD^\rb_{\rhol}(\DXS)$.

We~proceed by induction on the dimension of $Z_{\shm}$ (\cf\eqref{eq:strictperv2}). If $\dim Z_{\shm}=\nobreak0$, then $P_X(\shm)$ holds true by Kashiwara's equivalence and \ref{Luisa}\eqref{e}, since $P_X(\shm)$ obviously holds if $X$ has dimension zero.

Let us suppose $P_X(\shn)$ true for any $\shn\in \rD^\rb_{\rhol}(\DXS)$ such that $\dim Z_\shn<k$ (with $k\geq 1$) and let us prove the truth of $P_X(\shm)$ for $\shm \in \rD^\rb_{\rhol}(\DXS)$ with $\dim Z_\shm=k$.

By \ref{Luisa}\eqref{b} and \eqref{c}, we are reduced to proving $P_X(\shm)$ in the case where $\shm$ is a regular holonomic $\DXS$-module with $\dim Z_\shm=k$.

Following the notation of Section~\ref{subsec:prelimD}, let $ t(\shm)$ (respectively $f(\shm)$) be the torsion part (respectively the strict quotient) of $\shm$. According to \ref{Luisa}\eqref{c} (applied to the distinguished triangle $t(\shm)\to \shm \to f(\shm) \To{+1}$) and to \ref{Luisa}\eqref{g}, we are reduced to proving $P_X(f(\shm))$. Notice that $\dim Z_{f(\shm)}\leq k$ since $Z_{f(\shm)}\subseteq Z_\shm$. If $\dim Z_{f(\shm)}<k$, $P_X(f(\shm))$ holds true by induction. Hence we are reduced to proving $P_X(\shm)$ in the case where $\shm$ is a strict regular holonomic $\DXS$-module such that $\dim Z_\shm=k$, a property that we now assume to hold. Locally (recall that $P_X(\shm)$ is a local statement by \ref{Luisa}\eqref{a}), there exists a hypersurface $Y$ in~$X$ satisfying the assumptions of Proposition \ref{P:strictperv2}.

On the one hand, it is enough to check the property $P_X(\shm)$ for those $F\in\rD^\rb_{\cc}(\pOXS)$ such that $F=F\otimes\CC_{(X\setminus Y)\times S}$. Indeed, let us check that it holds for those $F$ such that $F=F\otimes\CC_{\YS}$. For any $F\in\rD^\rb_{\cc}(\pOXS)$, the complex $\shn:= \RH^S_X(F\otimes\CC_{\YS})\simeq R\Gamma_{[\YS]}(\RH^S_X(F))$ belongs to $\rD^\rb_{\rhol}(\DXS)$ according to Proposition~\ref{RHSloc}\eqref{RHSloc2}, and we have, by \cite[(3)]{MFCS1},
\[
\Rhom_{\DXS}(\shm,\shn)\simeq \Rhom_{\DXS}(\bD \shn, \bD \shm).
\]
The duality functor preserves $\rD^\rb_{\hol}(\DXS)$ by \cite[Prop.\,2.5]{SS1} and also $\rD^\rb_{\rhol}(\DXS)$ since it does so in the absolute case and $Li_s^*(\bD\shm)\simeq \bD(Li^*_s\shm)$. Let us also notice that $\bD\shm=\shh^0\bD\shm$ is strict holonomic (\cf\cite[Prop.\,2]{MFCS2}). Since $\shn$ has $\DXS$-coherent cohomology and is supported on $\YS$, we have
\[
\Rhom_{\DXS}(\bD\shn, (\bD\shm)(*(\YS)))=0.
\]
Furthermore, $\bD\shm$ being regular holonomic and strict, so is $(\bD\shm)(*(\YS))$ by Proposition \ref{P:strictperv2}, hence $R\Gamma_{[\YS]}(\bD\shm)$ is also regular holonomic, as well as $\shm':=\bD R\Gamma_{[\YS]}(\bD\shm)$. Finally, applying once more \cite[(3) \& (1)]{MFCS1}, we obtain
\[
\Rhom_{\DXS}(\shm,\shn)\simeq\Rhom_{\DXS}(\shm',\shn),
\]
with $\dim Z_{\shh^j\shm'}<k$ for any $j$, so the latter complex is $S$-$\CC$-constructible by the induction hypothesis.

On the other hand, $\shm(*(\YS))$ is regular holonomic, according to Proposition \ref{P:strictperv2}. We can now apply \ref{Luisa}\eqref{c} to the triangle
$
R\Gamma_{[\YS]}(\shm)\to \shm \to \shm(*(\YS))\To{+1}
$
(which is a distinguished triangle in $ \rD^\rb_{\rhol}(\DXS)$). By the induction hypothesis, $P_X(R\Gamma_{[\YS]}(\shm))$ holds true.

We thus assume that $\shm=\shm(*(\YS))$ is strict, and $F=F\otimes\CC_{(X\setminus Y)\times S}$. Let $\pi:X'\to X$ be as in Proposition \ref{P:strictperv2} and set $\delta=\dim X'-\dim X$. Note that the assumption on $F$ entails
\[
\pi^{-1}F=\pi^{-1}F\otimes\C_{(X'\setminus Y')\times S},
\]
while $\Dpi^*\shm[\delta]$ is concentrated in degree zero and is of D-type along~$Y'$. According to Lemma \ref{D-type 5}, $\Rhom_{\DXpS}(\Dpi^*\shm[\delta],\RH^S_{X'}(\pi^{-1}F))$ is an object of $\rD^\rb_{\cc}(p_{X'}^{-1}\sho_S)$, isomorphic to $\Rhom_{\DXpS}(\Dpi^*\shm,\Dpi^*\RH^S_X(F))$ by Proposition \ref{KSmorph}, and thus $R\pi_*$ of the latter is an object of $\rD^\rb_{\cc}(p_{X}^{-1}\sho_S)$. By~adjunction we have (\cf\cite[Th.\,4.33]{Ka2})
\begin{align*}
R\pi_*\Rhom_{\DXpS}(\Dpi^*\shm,\Dpi^*&\RH^S_X(F))\\
&\simeq\Rhom_{\DXS}(\Dpi_*\Dpi^*\shm[\delta],\RH^S_X(F))\\
&\simeq\Rhom_{\DXS}(\shm,\RH^S_X(F)),
\end{align*}
since the adjunction $\Dpi_*\Dpi^*[\delta]\to\id$ is an isomorphism when applied to $\DXS(*(\YS))$-modules. This ends the proof of Theorem \ref{Tequivtf}.
\end{proof}

\subsection{Proof of Corollary \ref{D2}}\label{subsec:proofcorD2}
For any $F\in\rD^\rb_{\cc}(\pOXS)$, we have functorial isomorphisms
\[
\pSol_X \bD(\RH^S_X(F))\simeq \bD\pSol_X (\RH^S_X(F))\simeq \bD F\simeq \pSol_X \RH^S_X(\bD F).
\]
Corollary \ref{D2} then follows by the full faithfulness of the functor $\pSol_X$.\qed

\backmatter
\providecommand{\hal}[1]{\href{https://hal.archives-ouvertes.fr/hal-#1}{\texttt{hal-#1}}}
\providecommand{\bysame}{\leavevmode\hbox to3em{\hrulefill}\thinspace}
\providecommand{\MR}{\relax\ifhmode\unskip\space\fi MR }
\providecommand{\MRhref}[2]{%
  \href{http://www.ams.org/mathscinet-getitem?mr=#1}{#2}
}
\providecommand{\href}[2]{#2}


\begin{thebibliography}{10}

\bibitem{EP}
M.~J. Edmundo and L.~Prelli, \emph{Sheaves on {$\mathcal{T}$}-topologies},
  J.~Math. Soc. Japan \textbf{68} (2016), no.~1, 347--381.

\bibitem{FMF1}
L.~Fiorot and T.~Monteiro~Fernandes, \emph{{$t$}-structures for relative
  {$\mathcal{D}$}-modules and {$t$}\nobreakdash-exactness of the de {R}ham
  functor}, J.~Algebra \textbf{509} (2018), 419--444.

\bibitem{Ka1}
M.~Kashiwara, \emph{{On the holonomic systems of differential equations II}},
  Invent. Math. \textbf{49} (1978), 121--135.

\bibitem{Ka3}
\bysame, \emph{The {Riemann-Hilbert} problem for holonomic systems}, Publ.
  RIMS, Kyoto Univ. \textbf{20} (1984), 319--365.

\bibitem{Ka2}
\bysame, \emph{{$D$}-modules and microlocal calculus}, Translations of
  Mathematical Monographs, vol. 217, American Mathematical Society, Providence,
  R.I., 2003.

\bibitem{K-K81}
M.~Kashiwara and T.~Kawai, \emph{{On the holonomic systems of differential
  equations (systems with regular singularities) III}}, Publ. RIMS, Kyoto Univ.
  \textbf{17} (1981), 813--979.

\bibitem{KS1}
M.~Kashiwara and P.~Schapira, \emph{Sheaves on manifolds}, Grundlehren Math.
  Wiss., vol. 292, Springer-Verlag, Berlin, Heidelberg, 1990.

\bibitem{KS4}
\bysame, \emph{Moderate and formal cohomology associated with constructible
  sheaves}, M{\'e}m. Soc. Math. France (N.S.), vol.~64, Soci{\'e}t{\'e}
  Math{\'e}matique de France, Paris, 1996.

\bibitem{KS5}
\bysame, \emph{Ind-sheaves}, Ast{\'e}risque, vol. 271, Soci{\'e}t{\'e}
  Math{\'e}matique de France, Paris, 2001.

\bibitem{KS3}
\bysame, \emph{{Categories and sheaves}}, Grundlehren Math. Wiss., vol. 332,
  Springer-Verlag, Berlin, Heidelberg, 2006.

\bibitem{KS6}
\bysame, \emph{{Regular and irregular holonomic D-modules}}, London
  Mathematical Society Lecture Notes Series, vol. 433, Cambridge University
  Press, 2016.

\bibitem{L-S87}
Y.~Laurent and P.~Schapira, \emph{Images inverses des modules
  diff{\'e}rentiels}, Compositio Math. \textbf{61} (1987), no.~2, 229--251.

\bibitem{Maisonobe16}
{\relax Ph}.~Maisonobe, \emph{Filtration relative, l’id{\'e}al de {B}ernstein
  et ses pentes}, 2016, \hal{01285562v2}.

\bibitem{Mochizuki11}
T.~Mochizuki, \emph{{Mixed twistor D-Modules}}, Lect. Notes in Math., vol.
  2125, Springer, Heidelberg, New York, 2015.

\bibitem{TL}
T.~Monteiro~Fernandes and L.~Prelli, \emph{Relative subanalytic sheaves}, Fund.
  Math. \textbf{226} (2014), no.~1, 79--100.

\bibitem{MFCS1}
T.~Monteiro~Fernandes and C.~Sabbah, \emph{{On the de Rham complex of mixed
  twistor $\mathcal{D}$-modules}}, Internat. Math. Res. Notices (2013), no.~21,
  4961--4984.

\bibitem{MFCS3}
\bysame, \emph{{Relative Riemann-Hilbert correspondence in dimension one}},
  Portugal. Math. \textbf{74} (2017), no.~2, 149--159.

\bibitem{MFCS2}
\bysame, \emph{{Riemann-Hilbert correspondence for mixed twistor
  $\mathcal{D}$-modules}}, J.~Inst. Math. Jussieu \textbf{18} (2019), no.~3,
  629--672.

\bibitem{Sabbah05}
C.~Sabbah, \emph{{Polarizable twistor $\mathcal{D}$-modules}}, Ast{\'e}risque,
  vol. 300, Soci{\'e}t{\'e} Math{\'e}matique de France, Paris, 2005.

\bibitem{SS1}
P.~Schapira and J.-P. Schneiders, \emph{Index theorem for elliptic pairs},
  Ast{\'e}risque, vol. 224, Soci{\'e}t{\'e} Math{\'e}matique de France, Paris,
  1994.

\end{thebibliography}
\end{document}